\documentclass[12pt]{article}

\usepackage{enumitem}
\usepackage{ifthen}
\newboolean{IncludeExerciseSolutions}
\setboolean{IncludeExerciseSolutions}{true}
\newboolean{IncludeTODOs}
\setboolean{IncludeTODOs}{true}
\usepackage{authblk}

\usepackage{xcolor}
\definecolor{mydarkblue}{rgb}{0,0.08,0.9}
\usepackage[colorlinks,linkcolor=blue,urlcolor=blue,citecolor=blue,pagebackref=true,backref=true]{hyperref}
\renewcommand*{\backrefalt}[4]{%
    \ifcase #1 \footnotesize{(Not cited.)}%
    \or        \footnotesize{(Cited on page~#2.)}%
    \else      \footnotesize{(Cited on pages~#2.)}%
    \fi}

\usepackage[utf8]{inputenc}
\usepackage{breakcites}

\usepackage{doi}
\usepackage{amsmath}
\usepackage{amsfonts}
\usepackage{amsthm}
\usepackage{amssymb}
\usepackage{amsmath}
\usepackage{amsfonts}
\usepackage{dsfont}
\usepackage{amsthm}
\usepackage{amssymb}
\usepackage[mathscr]{euscript}
\usepackage{nicefrac}
\usepackage{algorithm}
\usepackage{algpseudocode}
\usepackage{comment}
\usepackage{mathtools}

\usepackage{bm} 
\usepackage{dsfont} 
\usepackage{graphicx}
\usepackage{tikz-cd}
\usepackage{tikz}
\usepackage{float}

\usepackage{psfrag}
\usepackage{algorithm}
\usepackage{algpseudocode}
\algdef{SE}[DOWHILE]{Do}{doWhile}{\algorithmicdo}[1]{\algorithmicwhile\ #1}%
\usepackage{mathrsfs} 
\usepackage{sectsty}
\makeatletter\def\@seccntformat#1{\protect\makebox[0pt][r]{\csname the#1\endcsname\hspace{12pt}}}\makeatother
\definecolor{mycolor1}{rgb}{0.105882,0.619608,0.466667}
\definecolor{mycolor2}{rgb}{0.85098,0.372549,0.00784314}
\definecolor{mycolor3}{rgb}{0.458824,0.439216,0.701961}
\definecolor{mycolor4}{rgb}{0.905882,0.160784,0.541176}
\definecolor{mycolor5}{rgb}{0.4,0.65098,0.117647}
\definecolor{mycolor6}{rgb}{0.65098,0.462745,0.113725}
\definecolor{mycolor7}{rgb}{0.901961,0.670588,0.00784314}
\definecolor{mycolor8}{rgb}{0.4,0.4,0.4}
\definecolor{mycolor9}{rgb}{0.301961,0,0.294118}
\definecolor{mycolor10}{rgb}{0.0313725,0.25098,0.505882}

\usepackage{bigfoot}

\newcommand{\remove}[1]{}

\newcommand{\removesafe}[1]{}

\newcommand{\calP}{\mathcal{P}}

\newcommand{\ones}{\textbf{1}}

\newcommand{\calS}{\mathcal{S}}

\newcommand{\im}{\operatorname{im}}

\DeclareMathOperator{\diag}{diag}

\DeclareMathOperator{\Diag}{Diag}

\newcommand{\expect}{\mathbb{E}}

\newcommand{\frobnorm}[2][F]{\left\|{#2}\right\|_\mathrm{#1}}

\newcommand{\frobnormsmall}[1]{\|{#1}\|_\mathrm{F}}

\newcommand{\grad}{\nabla}
\newcommand{\hess}{\nabla^2}

\newcommand{\inner}[2]{\left\langle{#1},{#2}\right\rangle}

\newcommand{\innersmall}[2]{\langle{#1},{#2}\rangle}

\newcommand{\lambdamax}{\lambda_\mathrm{max}}
\newcommand{\lambdamin}{\lambda_\mathrm{min}}

\newcommand{\norm}[1]{\left\|{#1}\right\|}

\newcommand{\opnormsmall}[1]{\|{#1}\|_\mathrm{op}}

\newcommand{\PSD}{\mathrm{PSD}}

\newcommand{\rank}{\operatorname{rank}}

\newcommand{\reals}{{\mathbb{R}}}

\newcommand{\sigmamin}{\sigma_{\operatorname{min}}}

\newcommand{\Skew}{\operatorname{Skew}}

\newcommand{\spann}{\mathrm{span}}

\newcommand{\sqfrobnorm}[2][F]{\frobnorm[#1]{#2}^2}

\newcommand{\sqfrobnormsmall}[1]{\frobnormsmall{#1}^2}

\newcommand{\St}{\mathrm{St}}
\newcommand{\Cent}{\mathrm{Cent}}
\newcommand{\Holl}{\mathrm{Hol}}
\newcommand{\CPSD}{\mathrm{CPSD}}

\newcommand{\Sym}{\mathrm{Sym}}

\newcommand{\normal}{\mathrm{N}}
\DeclareMathOperator{\trace}{Tr}

\newcommand{\transpose}{^\top\! }

\newcommand{\aref}[1]{\hyperref[#1]{A\ref{#1}}}

\newtheorem{theorem}{Theorem}[section]

\newtheorem{corollary}[theorem]{Corollary}

\newtheorem{lemma}[theorem]{Lemma}

\newtheorem{proposition}[theorem]{Proposition}

\newtheorem{remark}[theorem]{Remark}

\usepackage[lmargin=3cm,rmargin=3cm,bottom=3cm,top=3cm]{geometry}
\usepackage[round,compress]{natbib}
\usepackage{sectsty}
\makeatletter\def\@seccntformat#1{\protect\makebox[0pt][r]{\csname the#1\endcsname\hspace{12pt}}}\makeatother
\usepackage{subcaption}
\usepackage{enumitem} 
\usepackage[normalem]{ulem} 
\usepackage{multicol}

\ifthenelse{\boolean{IncludeExerciseSolutions}}
{
	\newenvironment{answer}{ \color{blue} \paragraph{Answer.}}{\hfill$\blacksquare$\protect\\}
}
{
	
	\excludecomment{answer}
}
\ifthenelse{\boolean{IncludeTODOs}}
{
	\newcommand{\TODO}[1]{{\color{red}{[#1]}}}
}
{
	\newcommand{\TODO}[1]{}
}

\newcommand{\centeringmat}{P_c} 
\newcommand{\dimgt}{\ell} 
\newcommand{\dimopt}{k} 
\newcommand{\fu}{g} 
\newcommand{\fd}{f} 
\newcommand{\tildell}{\tilde \ell} 
\newcommand{\frob}{\mathrm{F}} 
\newcommand{\onevec}{\mathbf{1}} 
\newcommand{\edges}{E} 

\DeclarePairedDelimiterXPP{\normF}[1]{}{\lVert}{\rVert}{_{\mathrm{F}}}{#1} 

\newcommand{\gt}{Z_*} 
\newcommand{\pt}{Z} 
\newcommand{\gtgram}{Y_*} 
\newcommand{\ptgram}{Y} 
\newcommand{\holmat}{X} 

\newcommand{\erdosrenyi}{Erd\H{o}s--R\'enyi}

\title{Sensor network localization has a benign landscape after low-dimensional relaxation}
\author{{\fontsize{13.2pt}{12pt}\selectfont Christopher Criscitiello\thanks{The Wharton School, University of Pennsylvania, USA. \texttt{crisciti@wharton.upenn.edu}}, Andrew D.\ McRae\thanks{CERMICS, ENPC, Institut Polytechnique de Paris, CNRS, France. \texttt{andrew.mcrae@enpc.fr}}, Quentin Rebjock\thanks{Institute of Mathematics, EPFL, Lausanne, Switzerland. \texttt{quentin.rebjock@epfl.ch}}, Nicolas Boumal\thanks{Institute of Mathematics, EPFL, Lausanne, Switzerland. \texttt{nicolas.boumal@epfl.ch}}}}

\date{\today}

\begin{document}

\maketitle

\begin{abstract}
  We consider the sensor network localization problem, which is closely related to multidimensional scaling and Euclidean distance matrix completion.
  Given a ground truth configuration of $n$ points in $\reals^\dimgt$, we observe a subset of the pairwise distances and aim to recover the underlying configuration (up to rigid transformations).
  We show with a simple counterexample that the associated optimization problem is nonconvex and may admit spurious local minimizers, even when all distances are known.
  Yet, inspired by numerical experiments, we argue that all second-order critical points become global minimizers when the problem is relaxed by optimizing over configurations in dimension $\dimopt > \dimgt$.
  Specifically, we show this for two settings, both when all pairwise distances are known:
  (1) for arbitrary ground truth points, and $\dimopt = O(\sqrt{\dimgt n})$, and:
  (2) for isotropic random ground truth points, and $\dimopt = O(\dimgt + \log n)$.
To prove these results, we identify and exploit key properties of the linear map which sends inner products to squared distances.
\end{abstract}

\tableofcontents

\section{Introduction}
\label{sec:intro}
We consider the sensor network localization (SNL) problem.
For a ground truth configuration of $n$ points $z_1^*, \ldots, z_n^* \in \reals^\dimgt$, we observe noiseless pairwise Euclidean distances $d_{ij} = \|z_i^* - z_j^*\|$ for each $\{i, j\}$ in a set of edges $\edges$ between the points.
The goal is to recover the ground truth configuration from these distances.
One natural formulation of this task is as a quartic optimization problem:
\begin{align*}
  \min_{z_1, \ldots, z_n \in \reals^\dimopt} \, \fu(z_1, \ldots, z_n) && \textrm{ with } && \fu(z_1, \ldots, z_n) = \frac{1}{2} \sum_{\{i, j\} \in \edges} \Big(\|z_i - z_j\|^2 - d_{ij}^2\Big)^2,
  \tag{SNL}
  \label{eq:snl}
\end{align*}
where $\dimopt$ is the optimization dimension (the obvious choice is $\dimopt = \dimgt$).
The ground truth is identifiable only up to rigid transformations (translations, rotations, and reflections), since distances are invariant under such transformations.

SNL has applications in robotics (where points correspond to sensors), dimensionality reduction for data analysis, and chemistry, particularly in determining molecular conformations---see Section~\ref{sec:related-work}.
In these applications, $n$ is usually large while $\dimgt$ is small, e.g., $\dimgt \in \{2, 3\}$.

\subsection*{Motivation}

The problem~\eqref{eq:snl} is nonconvex and NP-hard in general when $\edges$ is sparse~\citep{eren2004rigidity}.
It may admit spurious local minimizers (i.e., local minima which are not global minima), even when the graph $\edges$ is complete (all pairwise distances are known): see Figure~\ref{fig:counterexample} which provides a simple example to that effect.
As a result, local search methods such as gradient descent or trust regions may \emph{a priori} fail to find a global minimizer.

However, when $\edges$ is dense (i.e., most distances are known), spurious local minimizers appear to be rare: in practice, local search methods often perform remarkably well, even with random initialization.
In contrast, when $\edges$ is sparse, these same methods frequently get trapped in spurious minima, even when recovery is information-theoretically possible.\footnote{In Erdős--Rényi graphs for example, recovery is information-theoretically possible when the edge density is slightly above the connectivity threshold~\citep{lewrigidity2022}.}
Supporting simulations are shown in Figure~\ref{fig:simulations12}, with additional details provided in Section~\ref{sec:numerics}.

\begin{figure}
  \centering
  \begin{minipage}[b]{0.495\textwidth}
    \includegraphics[width=\textwidth]{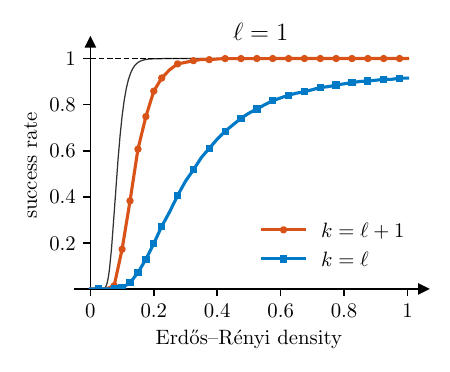}
  \end{minipage}
  \begin{minipage}[b]{0.495\textwidth}
    \includegraphics[width=\textwidth]{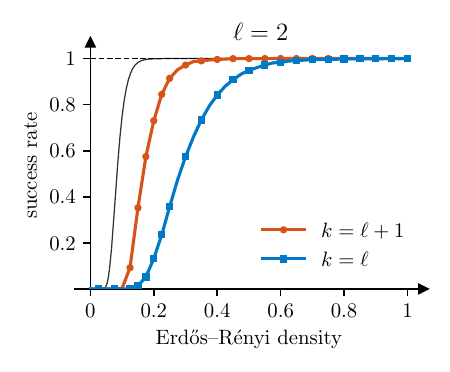}
  \end{minipage}
  \caption{Relaxing the dimension (that is, optimizing in dimension $\dimopt$ larger than the ground truth dimension $\dimgt$) empirically increases the probability that an optimization algorithm (here, trust regions) recovers the ground truth for \eqref{eq:snl} from a random initialization.
    This is apparent for all Erdős--Rényi densities, and especially for sparse graphs.
    The black curve marks the fraction of measurement graphs that are connected.
    Measurements here are noiseless; however, these empirical results are robust to moderate levels of noise (see Appendix~\ref{app:noise}).
    Details are in Section~\ref{sec:numerics}.}
  \label{fig:simulations12}
\end{figure}

To overcome these spurious local minimizers, we propose to \emph{relax} problem~\eqref{eq:snl} by allowing the points $z_1, \dots, z_n \in \reals^\dimopt$ to lie in a dimension $\dimopt > \dimgt$ higher than that of the ground truth.
The minimizers of the relaxed problem coincide with those of the original one as long as there are enough measurements, e.g., when $\edges$ is complete or, more generally, when it is universally rigid (see Section~\ref{sec:related-work}).

Crucially, we aim to keep $\dimopt$ as small as possible, since the relaxed problem involves optimizing over matrices of size $n \times \dimopt$. 
In contrast, standard methods based on convex relaxations (e.g., semidefinite programming) operate on dense $n \times n$ matrices, which becomes computationally prohibitive as $n$ grows. 
Therefore, we focus on \emph{low-dimensional nonconvex relaxations}, where $\dimopt \ll n$.

In many settings, we observe numerically that relaxing to even just $\dimopt = \dimgt + 1$ significantly improves the probability of recovery with a randomly initialized local search method---this also is depicted in Figure~\ref{fig:simulations12} with details in Section~\ref{sec:numerics}.
To explain this phenomenon, we adopt the perspective of \emph{landscape analysis}.
We say the objective function in~\eqref{eq:snl} has a \emph{benign landscape} if all its second-order critical points are global minima (see Section~\ref{subsec:derivatives}).
For such benign problems, several local search methods reliably find global minimizers~\citep{shub1987book,helmke1996optimization,chijinsaddles2017,jin2018agdescapes,jin2019escape,lee2019strictsaddles,cartis2012complexity,boumal2016globalrates}.

Notably, low-dimensional relaxations for~\eqref{eq:snl} empirically succeed in settings where standard theory provides no guarantees. 
In particular, a commonly invoked condition for benign landscapes in quartic problems is the restricted isometry property (RIP), which~\eqref{eq:snl} fails to satisfy---see Section~\ref{sec:propertiesEDM}. 

This discrepancy motivates the development of new theoretical tools to explain the observed behavior.
As part of our technical contributions, we identify and exploit properties of the sensing operator for SNL that are different from RIP and which explain (some of) the benign landscape results observed in practice.

\begin{figure}
  \centering
  \begin{minipage}[b]{0.49\textwidth}
    \includegraphics[width=\textwidth]{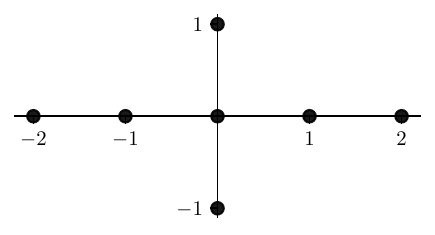}
  \end{minipage}
  \hfill
  \begin{minipage}[b]{0.49\textwidth}
    \includegraphics[width=\textwidth]{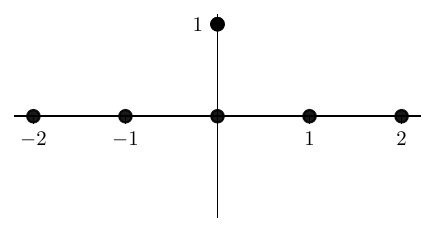}
  \end{minipage}
  \caption{Ground truth (left) and spurious configuration (right) for $n = 7$, $\dimgt = 2$, and $\dimopt = 2$.
    The latter has two overlapping points at $(0, 1)$.  Observe that when we relax to $\dimopt = 3$, it becomes possible to transform the spurious configuration (right) into the ground truth (left) while decreasing the cost~\eqref{eq:snl}: rotate one of the $(0,1)$ points around the horizontal axis---out of the plane---while keeping the other six points fixed.}
  \label{fig:counterexample}
\end{figure}

\subsection*{Overview of the results}

We focus on the complete graph and noiseless case---that is, all pairwise distances are known exactly.
In this setting, the problem is computationally simple and can be solved efficiently with an eigendecomposition (see Section~\ref{sec:related-work}).
However, we are primarily interested in the landscape of the optimization problem.
While our results do not fully explain the numerical observations mentioned above, they offer a foundation for further investigation.
In particular, they seem to constitute the first nontrivial theoretical guarantees on the landscape of~\eqref{eq:snl}.

For several years, it remained an open question whether~\eqref{eq:snl} with a complete graph admits a benign optimization landscape when $\dimopt = \dimgt$~\citep{malone2000sstress,parhizkar2013euclidean}. 
This question was recently resolved in the negative by~\citet[\S5]{song2024local}, who numerically identified spurious local minimizers for~\eqref{eq:snl} with $n = 100$ points, and then \emph{analytically} checked that configuration is stable by appealing to the Kantorovich theorem.
We add to this finding in Section~\ref{sec:counterexamples}, where we construct \emph{simple} counterexamples with  $n > 6$ points and $\ell=2$ (see Figure~\ref{fig:counterexample}), showing there exists a set of ground truth configurations (of positive measure) for which the landscape of~\eqref{eq:snl} is not benign, even when the graph $\edges$ is complete.\footnote{\label{counterexampleslit}Shortly after the release of this paper,~\citet{song2024local} independently uploaded an updated version of their paper where they included a simple counterexample with embedding dimension $\ell=1$ and any number of points $n > 6$.}
We further prove that for $\dimgt \geq 5$, the landscape can be non-benign even with only $n = \dimgt + 2$ points---which is the minimal number of points for such behavior.

Therefore, as is common in other contexts~\citep{sdplr}, we propose relaxing the optimization dimension to $\dimopt > \dimgt$ in order to obtain a benign landscape.
Simple arguments show that the landscape is benign with $\dimopt \geq n - 1$ --- see Section~\ref{sec:lowrank}.
Our goal, therefore, is to identify regimes where $\dimopt \ll n$ still suffices.

Our first results are \emph{deterministic}: see Section~\ref{sec:sqrt-n} for the precise statement and proof.
\begin{theorem}[All ground truths]\label{th:sqrtn-simplified}
  If $\edges$ in~\eqref{eq:snl} is complete and $\dimopt \gtrsim \dimgt + \sqrt{\dimgt n}$, then~\eqref{eq:snl} has benign landscape.
\end{theorem}
\noindent We emphasize that the \( \sqrt{n} \) term appearing in Theorem~\ref{th:sqrtn-simplified} does not arise from the Pataki--Barvinok rank bound, as it does in other landscape results~\citep{boumal2016bmapproach}.
It arises from a genuinely different argument.

We then turn to random ground truth configurations.
In this setting, a much lower dimensional relaxation suffices with \emph{high probability}: see Section~\ref{sec:isotropic} for the precise statement and proof.
\begin{theorem}[Random ground truths]\label{th:isotropic}
  Suppose the ground truth points are i.i.d.\ Gaussian random vectors with covariance $\Sigma \succ 0$.
  Let $\kappa = \lambda_{\max}(\Sigma) / \lambda_{\min}(\Sigma)$ be the condition number.
  If $\edges$ in~\eqref{eq:snl} is complete, $n$ is sufficiently large, and $\dimopt \gtrsim \kappa(\dimgt + \log n)$, then~\eqref{eq:snl} has benign landscape with high probability.
\end{theorem}

In Appendix~\ref{app:geomdescdirections}, we provide geometric interpretations of the descent directions constructed to prove both theorems.
We emphasize that our proof techniques used in both theorems are novel.
Moreover, our techniques for Theorem~\ref{th:isotropic} have subsequently been applied to other settings such as phase retrieval~\citep{mcrae2025phaseretrievalmatrixsensing}.

Numerical experiments in Section~\ref{sec:numerics} motivate us to conjecture that relaxing to just $\dimopt = \dimgt + 1$ is sufficient for the landscape to be benign when all distances are known, for any ground truth configuration.
We discuss the implications and several open questions arising from our observations in Section~\ref{sec:perspectives}.

\section{Related work}\label{sec:related-work}

\subsection*{Historical note}

Consider the problem of recovering positions (up to rigid transformations) in Euclidean space from all pairwise distances~\citep{shepard1962analysis, torgerson1958theory, kruskal1964multidimensional, kruskal1964nonmetric}.
When the distances are standard Euclidean distances, this problem can be solved efficiently.
The idea is to recover pairwise inner products from the distances, then to compute an eigendecomposition of the resulting Gram matrix.  See also~\citep{wickelmaier2003introduction, borg2007modern}.

The cost function~\eqref{eq:snl} is often referred to as the \emph{squared-stress} function (s-stress for short)~\citep{takane1977nonmetric,groenen1996least}.
Several variants exist, such as the simple stress function, in which norms and distances are not squared~\citep{deleeuw1993fitting}.
The latter is not differentiable everywhere and is known to have many spurious local minimizers~\citep{trosset1997existence}.
\cite{halsted2022riemannian} propose a Riemannian reformulation, with optimization over edge directions.

When some distances are unknown, recovering the points becomes significantly more difficult.
In this case, the problem is typically referred to as the ``localization problem'' or ``sensor network localization''~\citep{yemini1979some,zhang2010rigid,alrajeh2013localization}.
Often, one also constrains some of the points to fixed positions, i.e., they are anchors.
We do not focus on anchors in our work.

In general, recovering the ground truth configuration is hard, even under the assumption that it is unique.
An $\dimgt$-dimensional configuration is said to be \emph{globally rigid} if it is the unique realization (up to rigid transformations) of its pairwise distances in $\dimgt$-dimensional Euclidean space~\citep{asimow1978rigidity,asimow1979rigidity2,jackson2007notes}.
It is said to be \emph{universally rigid} if it is the unique realization in Euclidean spaces of \textit{any} dimension~\citep{gortler2014characterizing}.
For example, non-degenerate triangles are universally rigid.
Fairly sparse universally rigid graphs exist, such as trilateration graphs~\citep{zhu2010universal}.
Recovering a configuration is NP-hard~\citep{eren2004rigidity,aspnes2004computational,aspnes2006theory}, even when global rigidity is assumed.

\subsection*{Euclidean distance matrices}

A symmetric matrix containing pairwise squared distances is called a \emph{Euclidean distance matrix} (EDM)~\citep{parhizkar2013euclidean,liberti2014euclidean,dokmanic2015euclidean}.
\citet{fliege2019euclidean} survey optimization approaches based on EDMs to solve~\eqref{eq:snl}.
\citet{liberti2014euclidean, liberti2020distance, mucherino2012distance} survey Euclidean distance geometry and provide a broad list of applications, ranging from molecular conformation, wireless sensor networks and statics to dimensionality reduction and robotics.

\subsection*{Benign or not benign?}

It is a long-standing question whether the loss landscape of~\eqref{eq:snl} (and of other related costs) is benign when all distances are known.
\citet{trosset1997existence} exhibit a spurious configuration for a non-differentiable variant of the problem.
Subsequently, \citet{malone2000study} observe that this configuration is not a minimizer for~\eqref{eq:snl}, but rather a saddle point.
\citet{vzilinskas2003multimodality} identify a spurious local minimizer for~\eqref{eq:snl}, but with a non-realizable set of distances.
\citet[\S3.5]{parhizkar2013euclidean} also explores the existence of spurious configurations.

More recently, \citet{song2024local} analyze the landscape of~\eqref{eq:snl} and numerically find spurious local minimizers for~\eqref{eq:snl} with $\dimgt = 1$ and $\dimgt = 2$.
They analytically confirm that these configurations are truly spurious using sensitivity analysis. In Section~\ref{sec:counterexamples}, we provide \emph{simple} examples of spurious configurations (whose gradient and Hessian can be evaluated by hand).\textsuperscript{\ref{counterexampleslit}}

It is well known that increasing the number of parameters in a nonconvex problem can help eliminate spurious local minimizers and, in some cases, result in a benign landscape.
Among many other works, see for example \citet{sdplr, burer2005local, bandeira2016lowrankmaxcut}, and also \citet{safran2021effects} in the context of deep learning with teacher-student architectures.

In the setting of SNL, it has been observed \emph{empirically} that allowing the optimization to proceed in dimensions larger than the true embedding dimension can substantially improve numerical performance.
For instance,~\citet[\S 5]{fangoleary2011} propose a ``dimensionality relaxation’’ scheme in which one first attempts minimization in the target embedding dimension and, upon failure, progressively increases the optimization dimension.
Their underlying intuition is that lifting the configuration into a higher-dimensional space may create descent paths that bypass poor local minimizers.
Similarly,~\citet[\S IV]{TasissaLai2019ExactReconstruction} propose and numerically investigate a nonconvex augmented-Lagrangian method for SNL, reporting that optimizing in dimension $10$---even when the ground-truth points lie in $\mathbb{R}^3$---can significantly improve empirical performance.
Other papers which consider overparameterization in the context of SNL include~\citep{tangtoh2023} and~\citep[\S 6.2]{smith2024riemannian}.

\subsection*{Convex approaches}

The problem~\eqref{eq:snl} is nonconvex, but there exist several convex approaches to solve it.

One such approach is to phrase the recovery of the ground truth as a matrix sensing problem.
There exists a linear operator $\Delta \colon \Sym(n) \to \Sym(n)$ which, given a configuration $\pt$, maps the Gram matrix $\ptgram = \pt\pt^\top$ to the Euclidean distance matrix associated with $\pt$ (see Section~\ref{sec:mds-background} for details).
Let $\gt$ denote the ground truth configuration with $\gtgram = \gt^{}\gt^\top$ its Gram matrix, and let $M \in \{0, 1\}^{n \times n}$ be a binary mask such that $M_{ij} = 1$ if and only if the distance between $y_i$ and $y_j$ is known.
Then, we can attempt to recover $\gtgram$ by solving
\begin{align}\label{eq:mds-completion}
  \min_{\ptgram}\, \big\|\Delta(\ptgram - \gtgram) \odot M\big\|_\frob^2 && \text{subject to} && \ptgram \succeq 0,\, \rank \ptgram \leq \dimgt,
\end{align}
and factorizing the solution to retrieve $\gt$.
Here, $\|\cdot \|_\frob$ denotes the Frobenius norm, and $\odot$ denotes the entry-wise (Hadamard) product.

Problem~\eqref{eq:mds-completion} is nonconvex due to the rank constraint.
However, when all pairwise distances are known (or more generally when the ground truth is universally rigid), the rank constraint can be removed, and the resulting problem becomes convex (a quadratic semidefinite program).
Despite this, solving the convex formulation is computationally expensive, as it requires optimizing over dense $n \times n$ matrices.
Several works have explored convex algorithms to solve~\eqref{eq:mds-completion} and variants, including~\citep{bakonyi1995euclidian, alfakih1999solving, singer2010uniqueness}.

Another convex approach involves semidefinite programming with a linear cost (SDP).
In particular, polynomial-time algorithms based on SDP can recover universally rigid configurations (see definition in the historical note)~\citep{biswas2004semidefinite,biswas2006semidefinite,so2007theory,zhu2010universal}.
Drawing on results from the matrix completion literature,~\citet{TasissaLai2019ExactReconstruction} also formulate SNL as a SDP and show that, under coherence conditions, the underlying point cloud can be recovered from the SDP with high probability from a small number of random distance samples.
Unfortunately, SDP-based methods also involve optimizing over dense $n \times n$ matrices, with memory and computational costs that scale at least quadratically with $n$.

\subsection*{Nonconvex approaches}

Nonconvex methods for solving~\eqref{eq:mds-completion} and related problems have also been extensively studied; see, for example,~\citep{trosset1998new, trosset2000distance,ji2013beyond, ipsitaTasissa2024, smith2024riemannian, smith2025provablenonconvexeuclideandistance2025}.
In particular,~\citet{smith2024riemannian,smith2025provablenonconvexeuclideandistance2025} develop Riemannian gradient-descent algorithms with local linear convergence guarantees under suitable initialization.
They further provide provable spectral initialization schemes that ensure convergence with high probability when the sampling density is sufficiently large.

In contrast, the landscape of the nonconvex formulation~\eqref{eq:snl} is much less studied and understood.
This formulation is the Burer--Monteiro counterpart of~\eqref{eq:mds-completion}, in which the matrix $\ptgram$ is parameterized as $\pt\pt^\top$ (more on this in Section~\ref{sec:mds-background}).
As mentioned earlier, simple computations (see Section~\ref{sec:propertiesEDM}) show that~\eqref{eq:snl} does \emph{not} satisfy the restricted isometry property~\citep{ge2016matrix,bhojanapalli2016global,zhang2018howmuchrip}, which would otherwise imply that its landscape is benign.

\subsection*{A unique point to retrieve}

A notable variant of the SNL problem involves localizing a single point from its distances to a set of anchor points.
This is known as the GPS localization problem.
\cite{beck2008exact,beck2008iterative,beck2012solution} design algorithms based on a non-differentiable cost function, and \cite{sremac2019noisy} propose an SDP relaxation.

\section{Mathematical background for SNL}\label{sec:mds-background}

This section defines the basic objects underpinning SNL, and it introduces their important properties.
We focus on the complete graph case---where all pairwise distances are known---meaning that $\edges$ in~\eqref{eq:snl} contains all pairs $\{i, j\}$ for $i \neq j$.

\subsection{Notation}

The map $\diag \colon \reals^{n \times n} \to \reals^n$ extracts the diagonal of a matrix.
Overloading notation, we also use $\diag \colon \reals^n \to \reals^{n \times n}$ to denote the map that sends a vector $x$ to the diagonal matrix with diagonal $x$.
Further define $\Diag \colon \reals^{n \times n} \to \reals^{n \times n}$ as $\Diag \ptgram = \diag(\diag \ptgram)$: this map zeros out all off-diagonal entries of $\ptgram$.

Let $e_i \in \reals^n$ denote the $i$th standard basis vector, and let $\ones \in \reals^n$ denote the all-ones vector.
The $(i, j)$-entry of a matrix $X$ is $X_{ij}$.
The $\dimopt \times \dimopt$ identity matrix is denoted $I_\dimopt$.
The set of Stiefel matrices is denoted
$$\St(n, \dimopt) = \{V \in \reals^{n \times \dimopt} : V^\top V = I_\dimopt\},$$
while $\mathrm{O}(n) = \St(n,n)$ denotes the $n \times n$ orthogonal matrices.
The set of $n \times n$ skew-symmetric matrices is denoted $\Skew(n)$.
The set of $n \times n$ symmetric matrices is $\Sym(n)$.
We also use:
\begin{align*}
  \Cent(n) &= \big\{\ptgram \in \Sym(n): \ptgram \ones = 0\big\},\\
  \Holl(n) &= \big\{\holmat \in \Sym(n) : \diag \holmat = 0\big\},\\
  \PSD(n) &= \big\{ \ptgram \in \Sym(n) : \ptgram \succeq 0 \big\}, \\
  \PSD_{\leq \dimopt}(n) &= \big\{ \ptgram \in \Sym(n) : \ptgram \succeq 0, \, \rank \ptgram \leq \dimopt \big\}, \\
  \text{and}\quad \CPSD_{\leq \dimopt}(n) &= \big\{\ptgram \in \Cent(n) : \ptgram \succeq 0,\, \rank \ptgram \leq \dimopt\big\}.
\end{align*}
We usually denote \emph{centered} matrices (elements of $\Cent(n)$) by $\ptgram$, and \emph{hollow} ones (elements of $\Holl(n)$) by $\holmat$.
The set $\CPSD_{\leq \dimopt}(n)$ corresponds to the set of configurations of $n$ points in dimension $\dimopt$ after quotienting by rigid motion symmetries.

We always work with Frobenius inner products on these spaces:
\begin{align*}
  \langle X, \ptgram \rangle = \trace(X^\top \ptgram), && \sqfrobnormsmall{X} = \trace(X^\top X).
\end{align*}
For $X \in \Sym(n)$, we denote its smallest and largest eigenvalues by $\lambda_{\min}(X)$ and $\lambda_{\max}(X)$. 
For $X \in \mathbb{R}^{m \times n}$, we denote its 
$\min\{m, n\}$th largest singular value by $\sigma_{\min}(X)$, and its largest singular value by $\sigma_{\max}(X) = \opnormsmall{X}$.

The image of $X \in \reals^{m \times n}$, i.e., the span of its columns, is denoted $\im(X)$.
The kernel of $X$ is denoted $\ker(X)$.
The dimension of the kernel is denoted $\dim(\ker(X))$.

\subsection{The optimization problem}

Given points $z_1, \ldots, z_n$, their Gram matrix $\ptgram$ holds the inner products $\ptgram_{ij} = \inner{z_i}{z_j} = z_i^\top z_j^{}$.
The Euclidean distance matrix (EDM) for those points holds the squared distances $\|z_i - z_j\|^2 = \|z_i\|^2 + \|z_j\|^2 - 2\inner{z_i}{z_j}$.
From that expression, observe that the Gram matrix maps linearly to the EDM (up to a factor 2) via $\Delta \colon \Sym(n) \to \Holl(n)$, as follows:
\begin{equation*}\label{EDMmap}\tag{EDM--map}
\begin{split}
  \Delta(\ptgram) &= \frac{1}{2}\big(\ones \diag(\ptgram)^\top + \diag(\ptgram) \ones^\top - 2 \ptgram\big), \\
  \Delta(\ptgram)_{ij} &= \frac{1}{2}(\ptgram_{ii} + \ptgram_{jj} - 2 \ptgram_{ij}).
  \end{split}
\end{equation*}
Explicitly, $2 \Delta(\pt\pt^\top)$ contains the pairwise squared distances between the rows of $\pt$.
Therefore, we call $\Delta$ the \emph{EDM map}.

This allows us to rewrite~\eqref{eq:snl} in the following way.
Let $\gt \in \reals^{n \times \dimgt}$ be the ground truth configuration of $n$ points in dimension $\dimgt$.
Then, the SNL problem with complete graph (all distances are known) can be written in the form 
\begin{align*}\label{upstairs}\tag{SNL--complete}
  \min_{\pt \in \reals^{n \times \dimopt}} \, \fu(\pt) && \textrm{ with } && \fu(\pt) = \sqfrobnorm{\Delta\big(\pt \pt^\top - \gt^{} \gt^\top\big)},
\end{align*}
where $\dimopt$ is the optimization dimension, which may be larger than $\dimgt$.

The EDM of $\pt$ is invariant under translations $\pt \mapsto \pt + \ones z^\top, z \in \reals^\dimopt$; hence, so is the cost function $\fu$.
Accordingly, without loss of generality, we may restrict our attention to centered configurations, whose points have mean zero.
That is, we shall assume that both $\pt$ and $\gt$ are centered, when convenient:
\begin{align*}
  \pt^\top \ones = 0 && \text{and} && \gt^\top \ones = 0.
\end{align*}
A configuration $\pt$ is centered in this sense if and only if its Gram matrix $\ptgram = \pt \pt^\top$ is centered in the sense that $\ptgram \in \Cent(n)$, that is, $\ptgram \ones = 0$.

Relatedly, the map $\Delta$ is not invertible.
Indeed, it is easy to verify that the kernel of $\Delta$ contains $\{\ones v^\top + v \ones^\top : v \in \reals^n\}$.
Dimension counting, together with the eigenanalysis presented in Section~\ref{sec:propertiesEDM}, implies that the kernel of $\Delta$ actually equals $\{\ones v^\top + v \ones^\top : v \in \reals^n\}$.

However, the restriction of $\Delta$ to centered matrices, i.e., $\Delta \colon \Cent(n) \to \Holl(n)$, is invertible.
In particular, any centered global minimizer $\pt$ of~\eqref{upstairs}, potentially with $\dimopt > \dimgt$, satisfies $\pt \pt^\top = \gt \gt^\top$.

Throughout, we use $\Delta$ to denote its restriction to $\Cent(n)$, unless indicated otherwise.
Given a ground truth $\gt$ and a configuration $\pt$, we define the Gram matrices $\ptgram = \pt \pt^\top$ and $\gtgram = \gt^{} \gt^\top$.

\subsection{Properties of the EDM map}\label{sec:propertiesEDM}

Note that, for all $\pt, \dot \pt \in \reals^{n \times \dimopt}$, the definition of $\Delta$~\eqref{EDMmap} implies
\begin{align}\label{eq:identityinnerproduct}
  \Delta(\dot \pt \pt^\top + \pt {\dot \pt}^\top)_{ij} = \langle \dot z_i, z_i \rangle + \langle \dot z_j, z_j \rangle - \langle \dot z_i, z_j\rangle - \langle \dot z_j, z_i\rangle = \langle \dot z_i - \dot z_j, z_i - z_j\rangle.
\end{align}
This property plays an important role throughout, especially in Lemmas~\ref{prop:properties} and~\ref{lem:cute}, and in Proposition~\ref{generalcounterexample} for constructing configurations with non-benign landscape.

We let $\Delta^* \colon \Holl(n) \to \Cent(n)$ denote the adjoint of $\Delta$, which satisfies
\begin{align*}\label{eq:adjoint}\tag{EDM--adjoint}
  \Delta^*(\holmat) = \diag(\holmat \ones) - \holmat.
\end{align*}
Define the centering matrix $\centeringmat = I - \ones \ones^\top/n$, so that the orthogonal projector from $\Sym(n)$ to $\Cent(n)$ is $\ptgram \mapsto \centeringmat \ptgram \centeringmat$.
The following formulas for $\Delta^* \circ \Delta \colon \Cent(n) \to \Cent(n)$ and its inverse play a crucial role in our analysis (especially for Theorem~\ref{th:sqrtn-simplified}):
\begin{equation}\label{eq:deltastardelta}
  \begin{split}
    (\Delta^* \circ \Delta)(\ptgram) &= \ptgram + \frac{n}{2} \centeringmat\Diag(\ptgram)\centeringmat + \frac{1}{2}\trace(\ptgram) \centeringmat,\\
    (\Delta^* \circ \Delta)^{-1}(\ptgram) &= \ptgram - \centeringmat \Diag(\ptgram) \centeringmat.
  \end{split}
\end{equation}
See Appendix~\ref{app:formulas} for a brief derivation.

\paragraph{Eigenanalysis of $\Delta^* \circ \Delta$.} Equations~\eqref{eq:deltastardelta} reveal the eigenvalues and eigenspaces of $\Delta^* \circ \Delta$.
In particular, $\Delta^* \circ \Delta \colon \Cent(n) \to \Cent(n)$ has eigenvalues $1$, $\frac{n}{2}$, and $n$, with respective multiplicities $n(n - 3)/2$, $n - 1$, and 1.
The respective eigenspaces are: hollow centered matrices $\Holl(n) \cap \Cent(n)$, traceless diagonal matrices after centering $\{\centeringmat D \centeringmat : \trace(D) = 0 \text{ and $D$ is diagonal} \}$, and scalar multiples of $\centeringmat$.

From this, we see that $\ptgram \mapsto \langle \ptgram, (\Delta^* \circ \Delta)(\ptgram)\rangle$ 
is a strongly convex quadratic on $\Cent(n)$:
\begin{align}\label{eq:lowerboundonDeltacircDelta}
\langle \ptgram, (\Delta^* \circ \Delta)(\ptgram)\rangle = \sqfrobnorm{\Delta(\ptgram)} \geq \sqfrobnorm{\ptgram} \quad \quad \forall \ptgram \in \Cent(n).
\end{align}

The spectrum of $\Delta^* \circ \Delta$ was also recently characterized by~\citet{LICHTENBERG202486}.
Their approach makes use of a graph-theoretic result, providing a complementary perspective to ours, where the eigenvalues and eigenspaces emerge directly from the revealing formulas~\eqref{eq:deltastardelta}.\footnote{More precisely, here is how our notation connects to theirs.
Define the ``sensing basis'' $\{w_{ij}\}$ and ``dual basis'' $\{v_{ij}\}$ as in~\citep{LICHTENBERG202486}.
Then $w_{ij} = \Delta^*(q_{ij}), v_{ij} = \Delta^{-1}(q_{ij})$
up to small absolute constant factors, where $q_{ij} = \frac{1}{2} (e_i e_j^\top + e_j e_i^\top)$.
Moreover, their matrix $H$ is the matrix representation of $\Delta \Delta^*$ in the basis $\{q_{ij}\}$ for $\Holl(n)$.}

\paragraph{The restricted isometry property fails.} 
Problem~\eqref{upstairs} is an instance of \emph{low-rank matrix sensing}~\citep{Chi_2019} with sensing operator $\Delta$ given by~\eqref{EDMmap}.
A common tool for studying the landscape of matrix sensing problems is the Restricted Isometry Property (RIP).
An \emph{arbitrary} linear map $\tilde\Delta \colon \Cent(n) \to \Holl(n)$ is said to satisfy $(r, \delta_r)$-RIP if there exists a fixed scaling $c > 0$ such that for all matrices $\ptgram \in \Cent(n)$ of rank at most $r$ the following holds~\citep{ge2016matrix,bhojanapalli2016global}:
\begin{align}\label{eq:RIP}
(1-\delta_r)\sqfrobnorm{\ptgram} \leq c \langle \ptgram, (\tilde\Delta^* \circ \tilde\Delta)(\ptgram) \rangle \leq (1+\delta_r) \sqfrobnorm{\ptgram}.
\end{align}
It is known that if $\delta_{2\dimopt} \leq \frac{1}{5}$, then problem~\eqref{upstairs} with $\tilde\Delta$ instead of $\Delta$ admits a benign landscape, and that $\delta_{2\dimopt} < \frac{1}{2}$ is a necessary condition for establishing such results using the RIP framework~\citep{zhang2018howmuchrip}.

Let us show that $\Delta$ defined in~\eqref{EDMmap} fails to satisfy RIP with small $\delta$.
Observe that $\centeringmat (e_1^{} e_1^\top - e_2^{} e_2^\top) \centeringmat$ is rank $2$ with eigenvalue $\frac{n}{2}$, and 
$$(e_ 1 + e_ 2 - e_ 3 - e_ 4) (e_ 1 + e_ 2 - e_ 3 - e_ 4)^\top - (e_ 1 - e_ 2 - e_ 3 + e_ 4) (e_ 1 - e_ 2 - e_ 3 + e_ 4)^\top$$
is rank $2$ with eigenvalue $1$, because it is an element of the eigenspace $\Holl(n) \cap \Cent(n)$.
Therefore, 
$$\delta_{2\dimopt} \geq \delta_{2} \geq \frac{\frac{n}{2}-1}{\frac{n}{2}+1} = 1 - \frac{4}{n+2},$$
which implies that RIP does not hold with $\delta_{2\dimopt} < \frac{1}{2}$ and $n$ large.
Therefore, standard landscape analyses under RIP do not apply.

\subsection{The EDM map as a structured perturbation of the identity}

We have shown that the EDM map $\Delta$ is not well-conditioned in the sense of RIP, and yet it still seems to possess some special structure.
What is that structure?
Our approach is to view $\Delta^*\circ \Delta$ as a \emph{structured perturbation} of the identity.

To make this quantitative, define the linear maps $\Psi, \Gamma \colon \Cent(n) \to \Cent(n)$ by
\begin{align}
  (\Delta^* \circ \Delta)(\ptgram) = \ptgram + \Psi(\ptgram), \quad \quad (\Delta^* \circ \Delta)^{-1}(\ptgram) = \ptgram - \Gamma(\ptgram).
  \label{eq:defininggammapsi}
\end{align}
The structured perturbations are given by $\Psi$ and $\Gamma$.\footnote{The signs in equation~\eqref{eq:defininggammapsi} are chosen to reflect the identity $(1+t)^{-1} \approx 1 - t$ for $t \in \reals$ small.}

The cost function $\fu$ depends on $\Delta$ only through $\Delta^* \circ \Delta$, which is fully captured by either $\Psi$ or $\Gamma$. The proof of Theorem~\ref{th:sqrtn-simplified} relies on these maps only through the following properties. Therefore, in principle, the results of this paper  extend to sensing operators beyond $\Delta$ as long as they have the qualities below.
\begin{lemma}\label{prop:properties}
For all $\ptgram \in \Cent(n)$ such that $\ptgram \succeq 0$, we have:
\begin{enumerate} [label=\normalfont\textbf{P\arabic*}]
\item \label{P1} $\Gamma(\ptgram) \succeq 0$.

\item \label{P2} $\lambda_{\max}(\ptgram) > \lambda_{\max}(\Gamma(\ptgram))$ if $\ptgram \neq 0$.

\item \label{P3} $\trace(\ptgram) \geq \trace(\Gamma(\ptgram))$.

\item \label{P4} $\Psi(\ptgram) \preceq \frac{n}{2}\Gamma(\ptgram) + \frac{1}{2} \trace(\ptgram) \centeringmat$.

\item \label{P5} $\sqfrobnormsmall{\Delta(u w^\top + w u^\top)}  \leq 4 \langle \Delta(u u^\top), \Delta(w w^\top) \rangle$ for all $u, w \in \reals^n$. 
\end{enumerate}
\end{lemma}
Property~\ref{P1} says $\Gamma$ is completely positive.  Properties~\ref{P2} and~\ref{P3} say that $\Gamma$ is contractive in the operator and trace norms.  Property~\ref{P4} can be viewed as a nonstandard measure of the conditioning of $\Delta$.
Indeed, the $n$ that appears in~\ref{P4} dictates how much we must relax in Theorem~\ref{th:sqrtn-simplified}.
For example, if instead of~\ref{P4}, we had $\Psi(\ptgram) \preceq \frac{c}{2}\Gamma(\ptgram) + \frac{1}{2} \trace(\ptgram) \centeringmat$ for some $c > 0$, then relaxing by $\dimopt \approx \dimgt + \sqrt{c \dimgt}$ would be sufficient.
Property~\ref{P5} is useful for extracting information from the second-order criticality conditions (see Lemma~\ref{lem:cute}).

For the EDM map, Properties~\ref{P4} and~\ref{P5} in fact hold with equality, but, anticipating extensions, we state them in a weaker form.
Note that~\ref{P5} remains valid for incomplete graphs, provided that $\Delta$ is appropriately modified.

\begin{proof}
The revealing formulas~\eqref{eq:deltastardelta} imply 
\begin{align*}
\Gamma(\ptgram) = \centeringmat \Diag(\ptgram) \centeringmat, && \Psi(\ptgram) = \frac{n}{2} \centeringmat\Diag(\ptgram)\centeringmat + \frac{1}{2}\trace(\ptgram) \centeringmat.
\end{align*}
Properties~\ref{P1},~\ref{P3} and~\ref{P4} follow immediately from these expressions. 

For property~\ref{P2}, we immediately see that $\lambda_{\max}(\ptgram) \geq \lambda_{\max}(\Gamma(\ptgram))$.
Let us show the inequality is strict.
Assume $\ptgram \neq 0$, which implies $\lambda_{\max}(\ptgram) > 0$.
The all-ones vector $\ones$ is an eigenvector of $\ptgram$ with eigenvalue zero, since $\ptgram$ is centered.
Therefore, any eigenvector of $\ptgram$ associated to $\lambda_{\max}(\ptgram) > 0$ must be orthogonal to $\ones$, by orthogonality of eigenspaces of symmetric matrices.
In particular, no standard basis vector $e_i$ can be an eigenvector associated with
$\lambda_{\max}(\ptgram)$, since $\ones^\top e_i = 1 \neq 0$.
This implies $e_i^\top \ptgram e_i^{} < \lambda_{\max}(\ptgram) $ for all $i$, by the variational characterization of the maximum eigenvalue.
Therefore,
$$\lambda_{\max}(\Gamma(\ptgram)) = \lambda_{\max}(\centeringmat \Diag(\ptgram) \centeringmat) \leq \lambda_{\max}(\Diag(\ptgram)) = \max_i e_i^\top \ptgram e_i^{} < \lambda_{\max}(\ptgram).$$

Now let us show~\ref{P5}.
Using the identity~\eqref{eq:identityinnerproduct}, we have 
\begin{align*}
\sqfrobnormsmall{\Delta(u w^\top + w u^\top)} 
&=  \sum_{i,j=1}^n \Delta(u w^\top + w u^\top)_{ij}^2 
= \sum_{i,j=1}^n  (u_i - u_j)^2(w_i - w_j)^2 \\
&= \sum_{i,j=1}^n 2 \Delta(u u^\top)_{ij} \cdot 2 \Delta(w w^\top)_{ij}
= 4 \langle \Delta(u u^\top), \Delta(w w^\top) \rangle.
\end{align*}
This proof of~\ref{P5} also extends to incomplete graphs: one simply replaces the sum over all $i, j$ with a sum over the edge set $\edges$, and modifies \(\Delta\) accordingly.
\end{proof}
The following is a consequence of properties~\ref{P1} and~\ref{P2}.
\begin{lemma}\label{lem:P5}
If $\ptgram \in \Cent(n)$ and $\lambda_{\max}(\ptgram) > 0$, then $\lambda_{\max}(\ptgram) > \lambda_{\max}(\Gamma(\ptgram))$.
\end{lemma}
\begin{proof}
  Eigendecompose $\ptgram$ into $\ptgram = \ptgram_{P} + \ptgram_{N}$ with $\ptgram_P \succeq 0$ and $\ptgram_N \preceq 0$.
  Then,
  $$\lambda_{\max}(\ptgram) = \lambda_{\max}(\ptgram_P) \stackrel{(1)}{>} \lambda_{\max}(\Gamma(\ptgram_P)) \stackrel{(2)}{\geq} \lambda_{\max}(\Gamma(\ptgram_P) + \Gamma(\ptgram_N)) = \lambda_{\max}(\Gamma(\ptgram)).$$
  Inequality $\stackrel{(1)}{>}$ follows from property~\ref{P2} and $\ptgram_P \neq 0$.
  Inequality $\stackrel{(2)}{\geq}$ follows from property~\ref{P1} and linearity: $-\Gamma(\ptgram_N) = \Gamma(-\ptgram_N) \succeq 0$.
\end{proof}

\subsection{Derivatives and optimality conditions}\label{subsec:derivatives}

Standard computations yield the gradient and the Hessian of $\fu$ given by~\eqref{upstairs}:
\begin{equation}\label{eq:derivatives}
\begin{split}
  \grad \fu(\pt) &= 4 (\Delta^* \circ \Delta)(\ptgram - \gtgram) \pt,\\
  \hess \fu(\pt)[\dot \pt] &= 4\Big((\Delta^* \circ \Delta)(\pt \dot \pt^\top + \dot \pt \pt^\top) \pt + (\Delta^* \circ \Delta)(\ptgram - \gtgram) \dot \pt\Big), \quad \dot \pt \in \reals^{n \times \dimopt}.
  \end{split}
\end{equation}
A configuration $\pt$ is \emph{first-order critical (1-critical)} if $\nabla \fu(\pt) = 0$.
It is \emph{second-order critical (2-critical)} if also $\langle \dot \pt, \nabla^2 \fu(\pt)[\dot \pt]\rangle \geq 0$ for all $\dot \pt \in \reals^{n \times \dimopt}$.
The expressions in~\eqref{eq:derivatives} yield the following first- and second-order criticality conditions for $\pt$:
\begin{align}
  &\textbf{(1-criticality)} \qquad (\Delta^* \circ \Delta)(\ptgram - \gtgram) \pt = 0,\label{eq:1crit}\\
  &\textbf{(2-criticality)} \qquad \big\langle \dot \pt {\dot \pt}^\top,  (\Delta^* \circ \Delta)(\ptgram - \gtgram) \big\rangle + \frac{1}{2} \sqfrobnormsmall{\Delta(\dot \pt \pt^\top + \pt \dot \pt^\top)} \geq 0,\label{eq:2crit}
\end{align}
for all $\dot \pt \in \reals^{n \times \dimopt}$.
To prove that the landscape of $\fu$ is benign, we must show that all 2-critical configurations $\pt$ are global minima, i.e., $\fu(\pt) = 0$.
We further note that the first term in the second-order condition~\eqref{eq:2crit} coincides (up to sign) with the \emph{quadratic energy} introduced by~\citet{connelly1982rigidity}.

\paragraph{Useful consequence of 2-criticality.}
Lemma~\ref{lem:cute} below is a consequence of~\eqref{eq:2crit} and property~\ref{P5}.
It is used in the proofs of both Theorems~\ref{th:sqrtn-simplified} and~\ref{th:isotropic}.
We emphasize that this lemma can easily be extended to incomplete graphs: the proofs are chosen to reflect this, in anticipation of future work on the landscape of SNL.
\begin{lemma}
	\label{lem:cute}
	Let $\pt \in \reals^{n \times \dimopt}$ satisfy $\langle \dot \pt, \hess \fu(\pt) \dot \pt \rangle \geq 0$ for all $\dot \pt \in \reals^{n \times \dimopt}$.
	For all $S \in \PSD(n)$ and all $T \in \PSD(\dimopt)$, we have 
	\begin{equation}
		\label{eq:cute_1}
		0 \leq \trace(T) \innersmall{S}{(\Delta^* \circ \Delta)(\pt \pt\transpose - \gt \gt\transpose)} + 2 \innersmall{\Delta(\pt T \pt\transpose)}{\Delta(S)}.
	\end{equation}
\end{lemma}
\begin{proof}[Proof of Lemma~\ref{lem:cute}]
Eigendecompose $S$ and $T$ as follows: $S = \sum_{i=1}^{n} u_i u_i^\top$ 
and $T = \sum_{j=1}^k v_j v_j^\top$. Plug $\dot \pt = u_i v_j^\top$ into the 2-criticality condition~\eqref{eq:2crit}, then use property~\ref{P5} with $u = u_i, w = \pt v_j$ 
to obtain
\begin{align*}
0 \leq& \innersmall{\dot \pt \dot \pt^\top}{(\Delta^* \circ \Delta)(\pt \pt\transpose - \gt \gt\transpose)} + \frac{1}{2}\sqfrobnormsmall{\Delta(\pt \dot \pt^\top + \dot \pt \pt^\top)} \\
\leq&
\|v_j\|^2 \innersmall{u_i u_i^\top}{(\Delta^* \circ \Delta)(\pt \pt\transpose - \gt \gt\transpose)} 
+ 2 \langle \Delta(\pt v_j v_j^\top \pt^\top), \Delta(u_i u_i^\top)  \rangle.
\end{align*}
Summing these inequalities over all $i$ and $j$ yields equation~\eqref{eq:cute_1}.
\end{proof}

We provide a second proof of Lemma~\ref{lem:cute} based on a \emph{randomized} set of descent directions.
This proof is useful for deriving a geometric interpretation of the descent directions used in Theorem~\ref{th:isotropic} --- see Appendix~\ref{app:geomdescdirections}.
\begin{proof}[Proof of Lemma~\ref{lem:cute} based on randomized directions]
Let $\dot \pt = A G B$, where $G$ is an $n \times \dimopt$ matrix of i.i.d.\ standard Gaussian random variables, and $A \in \reals^{n \times n}, B \in \reals^{\dimopt \times \dimopt}$ shall be specified later.
By assumption, we have $\expect\big[\langle \dot \pt, \nabla^2 \fu(\pt)[\dot \pt] \rangle\big] \geq 0$, where the expectation is over the randomness in $G$.

Note that $\expect\big[\dot \pt \dot \pt^\top\big]  = A \expect\big[G B B^\top G^\top\big] A^\top = \trace(B B^\top) A A^\top$, and so
\begin{align}\label{random1}
\expect\big[\innersmall{\dot \pt \dot \pt^\top}{(\Delta^* \circ \Delta)(\pt \pt\transpose - \gt \gt\transpose)}\big]
=
\trace(B B^\top)\innersmall{A A^\top}{(\Delta^* \circ \Delta)(\pt \pt\transpose - \gt \gt\transpose)}.
\end{align}
Let $a_i$ denote the $i$th row of $A$.
Equation~\eqref{eq:identityinnerproduct} implies 
$$\Delta\Big(A (G B \pt^\top) + (G B \pt^\top)^\top A\Big)_{ij} = \langle a_i - a_j, G B (z_i - z_j)\rangle.$$
Using this identity as well as $\expect \langle a, G b \rangle^2 = a^\top \expect\big[G b b^\top G^\top\big] a = \|a\|^2 \|b\|^2$, we find
\begin{align*}
\expect\sqfrobnormsmall{\Delta(\dot \pt \pt^\top + \pt \dot \pt^\top)}
&=
\sum_{i, j=1}^n \expect\langle a_i - a_j, G B (z_i - z_j)\rangle^2 \\
&=
\sum_{i, j=1}^n \|a_i - a_j\|^2 \| B (z_i - z_j)\|^2
= 4 \langle \Delta(A A^\top), \Delta(\pt B B^\top \pt^\top) \rangle.
\end{align*}
Using $\expect\big[\langle \dot \pt, \nabla^2 \fu(\pt)[\dot \pt] \rangle\big] \geq 0$ and formula~\eqref{eq:2crit}, we conclude
\begin{align*}
0 \leq& \trace(B B^\top)\innersmall{A A^\top}{(\Delta^* \circ \Delta)(\pt \pt\transpose - \gt \gt\transpose)} 
+ 2 \langle \Delta(A A^\top), \Delta(\pt B B^\top \pt^\top) \rangle.
\end{align*}
Setting $A = S^{1/2}$ and $B = T^{1/2}$ yields the lemma.
\end{proof}

\subsection{$2$-critical rank-deficient configurations are optimal}\label{sec:lowrank}

By optimizing over centered Gram matrices, we can equivalently write~\eqref{upstairs} as
\begin{align}\label{eq:downstairs}
  \min_{\ptgram} \, \fd(\ptgram) && \textrm{ with } && \fd(\ptgram) = \sqfrobnormsmall{\Delta(\ptgram - \gtgram)} && \textrm{ subject to } && \ptgram \in \CPSD_{\leq \dimopt}(n).
\end{align}
The cost function $\fd$ is convex and, when $\dimopt \geq n - 1$, the constraint set $\CPSD_{\leq \dimopt}(n)$ equals $\Cent(n) \cap \PSD(n)$, and so is also convex.
Therefore, \eqref{eq:downstairs} is a convex optimization problem {when $\dimopt$ is large enough}.
Furthermore, using arguments similar to those in~\citep[Prop.~2.7,~3.32]{levin2022lifts}, we find that second-order critical points of~\eqref{upstairs} map to stationary points of~\eqref{eq:downstairs}.
It follows that second-order critical points of~\eqref{upstairs} are global minima when $\dimopt \geq n - 1$.
This is also shown by~\citet{song2024local}, through a different proof. 

Accordingly, this paper aims to show that 2-critical configurations are globally optimal in cases where $\dimopt \leq n - 2$.
As a useful step in this direction, the general lemma below shows that rank-deficient 2-critical configurations are optimal.
The proof uses standard arguments that resemble those of~\citet[Thm.~7]{journee2010low}.
As a result, in all that follows it is sufficient to analyze full-rank 2-critical configurations.

Geometrically, the lemma says that if all points $z_i$ are contained in a proper subspace of $\reals^\dimopt$, then they either correspond to the ground truth or are not 2-critical.

\begin{lemma}\label{rankdeficiency}
  If $\pt \in \reals^{n \times \dimopt}$ is second-order critical for~\eqref{upstairs} and $\rank \pt < \dimopt$, then $\pt$ is a global minimizer of~\eqref{upstairs}.
\end{lemma}
\begin{proof}
Let $\pt \in \reals^{n \times \dimopt}$ be 2-critical for~\eqref{upstairs}, and define $\ptgram = \pt \pt^\top$.
Since $\fu(\pt) = \fd(\pt \pt^\top) = \fd(\ptgram)$, a result in~\citep[Prop.~3.32]{levin2022lifts} directly implies that $\ptgram$ is a first-order critical point of the problem
$\min_{\ptgram \in \PSD_{\leq \dimopt}(n)} \, \fd(\ptgram) = \|\Delta(\ptgram - \gtgram)\|^2,$
which is the same as~\eqref{eq:downstairs} but we ignore centering.
This means that $-\nabla \fd(\ptgram)$ is in the Fr\'echet normal cone $\normal_\ptgram \PSD_{\leq \dimopt}(n)$~\citep[\S3.3, \S3.4]{ruszczynski2006nonlinear}.

When $\rank \pt < \dimopt$, the normal cones $\normal_\ptgram \PSD(n)$ and $\normal_\ptgram \PSD_{\leq \dimopt}(n)$ coincide~\citep[Cor.~4.12]{levin2022lifts}.
Therefore, $\ptgram$ is also first-order critical (thus, optimal) for the convex problem $\min_{\ptgram \succeq 0} \, \fd(\ptgram)$.
In particular, $\ptgram$ is a global minimum for~\eqref{eq:downstairs}.
Hence, $\pt$ is a global minimum for~\eqref{upstairs}.
\end{proof}

\section{\boldmath Landscape for all ground truths, $\dimopt \approx \sqrt{n}$}\label{sec:sqrt-n}

In this section, we prove the following more precise statement of Theorem~\ref{th:sqrtn-simplified}.
Recall that the ground truth $\gt$ lives in dimension $\dimgt$ while we optimize for $\pt$ in dimension $\dimopt$.

\begin{theorem}\label{th:sqrt-n}
If $\pt \in \reals^{n \times \dimopt}$ is second-order critical for~\eqref{upstairs}, and
  \begin{align}\label{assumptiononk}
    \dimopt > \dimgt + \frac{n}{-1/2 + \sqrt{1/4 + n / \dimgt}} \approx \dimgt  + \sqrt{n \dimgt},
  \end{align}
  then $\pt$ is a global minimizer.
\end{theorem}
As is apparent from the proof below, Theorem~\ref{th:sqrt-n} holds for any map $\Delta$ satisfying properties~\ref{P1} to~\ref{P5}.
In Appendix~\ref{app:otherconditions}, we show that the theorem also holds under a different set of conditions on $\Delta$, and give a recipe for generating maps $\Delta$ which satisfy that set of conditions.

\subsection{Exploiting first-order criticality}

Assume $\pt$ is $1$-critical for~\eqref{upstairs} and full rank.\footnote{For $\pt \in \reals^{n \times \dimopt}$, ``full rank'' means $\rank \pt = \dimopt$. If instead $\rank \pt = n < \dimopt$, then Lemma~\ref{rankdeficiency} already ensures that 2-criticality implies optimality.}
As mentioned previously, we may assume $\pt$ and $\gt$ are centered throughout.
Define the Gram matrices $\ptgram = \pt \pt^\top, \gtgram = \gt \gt^\top$.
Use a polar factorization to write $\pt = V \Sigma^{1/2}$ with $V \in \St(n, \dimopt)$ and $\Sigma \succ 0$.
Pick $V_\perp$ so that $[V, V_\perp] \in \mathrm{O}(n)$, i.e., $V V^\top + V_\perp^{} V_\perp^\top = I_n$.
Since $\pt$ is centered,
\begin{align}\label{eq:Vsandstuff}
  \ptgram = V \Sigma V^\top, && \ones^\top V = 0, && \text{and} && V_\perp^{} V_\perp^\top \ones = \ones.
\end{align}
The 1-criticality condition~\eqref{eq:1crit} is conveniently expressed as\footnote{According to formula~\eqref{eq:adjoint}, $C$ is the Laplacian of the graph whose edge $\{i, j\}$ has weight $\frac{1}{2}\|z_i - z_j\|^2 - \frac{1}{2}\|z_i^* - z_j^*\|^2$.}
\begin{align}
  \textbf{(1 criticality)} \quad C\pt = 0 && \textrm{ with } && C = (\Delta^* \circ \Delta)(\ptgram - \gtgram) \in \Cent(n).
  \label{1criticality-and-C}
\end{align}
In order to prove the theorem, we must use the fact that $\gt$ is a configuration in $\reals^\dimgt$, that is, $\gtgram \in \mathrm{CPSD}_{\leq \dimgt}(n)$.
Therefore, let us manipulate the definition of $C$ and recall the definition of $\Gamma$~\eqref{eq:defininggammapsi} in order to isolate $\gtgram$:\footnote{This technique of isolating $Y_*$ is inspired by the classical landscape analysis in the special case where the sensing operator is the identity~\citep{boumal2023lworank1,boumal2023lworank2}.}
\begin{equation} \label{1criticalityrearranged}
\begin{split}
\gtgram &= \ptgram - (\Delta^* \circ \Delta)^{-1}(C) 
  = \ptgram - C + \Gamma(C).
\end{split}
\end{equation}

The cost of~\eqref{upstairs} at $\pt$ is $\fu(\pt) = \langle \ptgram - \gtgram, C \rangle$, which equals $-\langle \gtgram, C \rangle$ if $\pt$ is 1-critical.
Therefore, if $\pt$ is not optimal (i.e., $\fu(\pt) > 0$), then $C$ is not positive semidefinite.
In fact, the following lemma shows that $C$ is negative semidefinite.
\begin{lemma} \label{lemmaCNSD}
  If $\pt$ is first-order critical for~\eqref{upstairs}, then $C \preceq 0$.
  %
\end{lemma}

\begin{proof}
Let $\xi$ be a unit eigenvector of $C$ associated with its maximal eigenvalue 
$\gamma := \lambda_{\max}(C)$.
If $\gamma = 0$, the conclusion follows immediately, so suppose $\gamma \neq 0$.
The 1-criticality condition~\eqref{1criticality-and-C} then gives 
$0 = \pt(\pt^\top C)\xi = \ptgram C \xi = \gamma\, \ptgram \xi,$
i.e., $\ptgram \xi = 0$.
Using~\eqref{1criticalityrearranged} together with $\gtgram = \gt\gt^\top \succeq 0$, we obtain
\[
0 \le \xi^\top \gtgram \xi
= - \xi^\top C \xi + \xi^\top \Gamma(C)\xi
= -\gamma + \xi^\top \Gamma(C)\xi .
\]
Rearranging yields
$\lambda_{\max}(C) = \gamma \le \xi^\top \Gamma(C)\xi \le \lambda_{\max}(\Gamma(C)).$
By the contrapositive of Lemma~\ref{lem:P5}, we obtain $C \preceq 0$.
\end{proof}

The matrix $-C$ can be interpreted as a \emph{stress matrix} in the sense of rigidity theory.
If one could show that $-C$ has maximal rank, then Lemma~\ref{lemmaCNSD} would imply that the associated framework with vertex positions $Z$ is universally rigid~\citep{connelly2005generic,gortlerHealyThurston2010GlobalRigidity,gortler2014characterizing}.
In the present setting, however, this conclusion is somewhat vacuous, since the measurement graph is complete and frameworks on complete graphs are trivially universally rigid.

\subsection{Second-order criticality and descent directions}\label{sec:p2crit}

So far we assumed $\pt = V\Sigma^{1/2}$ is full rank and 1-critical, that is, $C \pt = 0$.
Due to centering, we also have $C \ones = 0$ with $\ones^\top V = 0$.
This implies that $\rank(C) \leq n-\dimopt-1$ and that the image of $C$ is contained in the image of $V_\perp$ (because $C$ is symmetric).
Moreover, $C \preceq 0$ by Lemma~\ref{lemmaCNSD}.
Let $\xi_1, \ldots, \xi_{n-\dimopt}$ be orthonormal eigenvectors of $C$ forming a basis for the image of $V_\perp$, with $\xi_{n-\dimopt} = \ones / \sqrt{n}$.
Order them so that the associated eigenvalues are $\gamma_1 \leq \cdots \leq \gamma_{n-\dimopt} = 0$.
Lastly, define
$$\tildell = \min\{\dimgt, n-k\}.$$ 
When $n$ is large, we simply have $\tildell = \dimgt$.

Next, let us assume $\pt$ is also 2-critical.
Recall the consequence of 2-criticality stated in Lemma~\ref{lem:cute}: it provides an inequality for each choice of $T, S$ positive semidefinite.
Plug $T = v v^\top$ for some $v \in \reals^\dimopt$ and 
\begin{align} \label{defS}
  S = - \sum_{i=1}^{\tildell} \gamma_i^{} \xi_i^{} \xi_i^\top \succeq 0
\end{align}
into that lemma to obtain the following lemma.
Note that $S$ is a truncation of $-C = - \sum_{i=1}^{n-\dimopt} \gamma_i^{} \xi_i^{} \xi_i^\top$ to rank at most $\tildell$.
\begin{lemma}\label{derivationdescent}
Let $\pt = V\Sigma^{1/2} \in \reals^{n \times \dimopt}$ be a full-rank, 2-critical point for~\eqref{upstairs}.
For every vector $u \in \ker(V^\top \gtgram V) \subseteq \reals^\dimopt$, we have
\begin{align}
    0 &= \Sigma u - P u, \quad \text{and} \label{SigmaUperp2} \\
    0 &\leq u^\top \left[- \langle S, S\rangle P + \trace(S) P^2 + nP^3 \right] u, \label{SigmaUperp}
\end{align}
with $S$ as in~\eqref{defS} and $P = V^\top \Gamma(-C) V \succeq 0$ with $C$ as in~\eqref{1criticality-and-C}.
\end{lemma}
Notice that $\ker(V^\top \gtgram V)$ is a subspace of dimension $\dimopt - \rank(\gtgram) \geq \dimopt - \dimgt$.
Here, we see plainly the benefit of relaxing the optimization dimension $\dimopt$ to be strictly larger than the ground truth dimension $\dimgt$: the larger the gap, the more inequalities are provided by~\eqref{SigmaUperp}.
Each of those exposes a candidate descent direction for us to use in the landscape analysis.
This insight guides the next parts.
\begin{proof}
  We choose $v \in \reals^\dimopt$ momentarily.
  Plug the above values for $T$ and $S$ into Lemma~\ref{lem:cute} to obtain inequality $\stackrel{(0)}{\leq}$ below then proceed to derive:
\begin{equation}\label{eq:usefulguy2}
\begin{split}
\|v\|^2 \langle S, S \rangle &= \|v\|^2 \langle S, -C \rangle \stackrel{(0)}{\leq} 2 \langle \pt v v^\top \pt^\top, (\Delta^* \circ \Delta)(S) \rangle
= 2 \langle \pt v v^\top \pt^\top, S + \Psi(S) \rangle \\
&\stackrel{(1)}{=} 2 \langle \pt v v^\top \pt^\top, \Psi(S) \rangle
\stackrel{(2)}{\leq} \langle \pt v v^\top \pt^\top, n \Gamma(S) + \trace(S) \centeringmat \rangle \\
&= \trace(S) v^\top \pt^\top \pt v + n v^\top \pt^\top \Gamma(S) \pt v \\
&\stackrel{(3)}{\leq} \trace(S) v^\top \pt^\top \pt v + n v^\top \pt^\top \Gamma(-C) \pt v.
\end{split}
\end{equation}
Equality $\stackrel{(1)}{=}$ is because $S\pt = 0$ owing to $C \pt = 0$ (1-criticality).  
Inequality $\stackrel{(2)}{\leq}$ follows from property~\ref{P4}.
Inequality $\stackrel{(3)}{\leq}$ follows from $C + S \preceq 0$ 
and property~\ref{P1}.

From here on, it is convenient to use the notation $\pt = V \Sigma^{1/2}$~\eqref{eq:Vsandstuff}.
Let $v = \Sigma^{1/2} u$ with $u \in \ker(V^\top \gtgram V)$.
Then inequality~\eqref{eq:usefulguy2} becomes
\begin{equation}\label{bleeblooeq2}
\begin{split}
\langle S, S \rangle u^\top \Sigma u &\leq 
\trace(S) u^\top \Sigma^2 u + n u^\top \Sigma V^\top \Gamma(-C) V \Sigma u \\
&= \trace(S) u^\top \Sigma^2 u + n u^\top \Sigma P \Sigma u,
\end{split}
\end{equation}
where we introduced $P = V^\top \Gamma(-C) V$.
Multiplying equation~\eqref{1criticalityrearranged} on the right and left by $V$ gives
$$V^\top \gtgram V = \Sigma - V^\top C V + V^\top \Gamma(C) V = \Sigma + V^\top \Gamma(C) V = \Sigma - P,$$
where the second-to-last equality follows from 1-criticality $0 = C \pt = C V \Sigma^{1/2}$.
Since $u \in \ker(V^\top \gtgram V)$,
$$0 = V^\top \gtgram V u = \Sigma u - P u.$$
Using this identity, equation~\eqref{bleeblooeq2} becomes
$$\langle S, S \rangle u^\top P u \leq 
\trace(S) u^\top P^2 u + n u^\top P^3 u,$$
as claimed.
\end{proof}

\subsection{Controlling the eigenvalues of $P$}

In light of Lemma~\ref{derivationdescent}, next we control the eigenvalues of $P = V^\top \Gamma(-C) V$.
The following lemma is a consequence of properties~\ref{P1} to~\ref{P3}, $\rank(\gtgram) \leq \dimgt$, $- C \succeq 0$, and first-order criticality (second-order criticality is not used).
\begin{lemma}\label{traceP}
  Let $\pt = V\Sigma^{1/2}$ be full rank and 1-critical for~\eqref{upstairs}.  We have
  \begin{equation}\label{boundtraceP}
    \begin{split}
      \trace(P) = \trace\big(V^\top \Gamma(-C) V\big) \leq \trace(S),
    \end{split}
  \end{equation}
  with $S$ as in~\eqref{defS} and $P = V^\top \Gamma(-C) V \succeq 0$ with $C$ as in~\eqref{1criticality-and-C}.
\end{lemma}
\begin{proof}
From equation~\eqref{1criticalityrearranged}, Lemma~\ref{lemmaCNSD}, and property~\ref{P1}, $\gtgram = \ptgram - C - \Gamma(-C) \preceq \ptgram - C.$
Since $V_\perp^\top \pt = 0$ (by definition of $V_\perp$) and $\ptgram = \pt \pt\transpose$, it follows that
\begin{align}\label{PSDness}
V_\perp^\top \gtgram V_\perp \preceq - V_\perp^\top C V_\perp.
\end{align}
Because $\rank(V_\perp^\top \gtgram V_\perp) \leq \min\{\dimgt, n-k\} = \tildell$, $\trace(V_\perp^\top \gtgram V_\perp)$ is at most the sum of the $\tildell$ largest eigenvalues of $-V_\perp^\top C V_\perp$.
On the other hand, the 1-criticality condition $0 = C \pt = C V \Sigma^{1/2}$ tells us that $C$ and $V_\perp^\top C V_\perp$ have exactly the same $n-\dimopt$ least eigenvalues which we previously denoted by $\gamma_i$ (the remaining $\dimopt$ eigenvalues are zero).
Therefore, 
\begin{align}\label{thisguy3}
  \trace(V_\perp^\top \gtgram V_\perp) \leq -(\gamma_1 + \cdots + \gamma_{\tildell}) = \trace(S),
\end{align}
where $S$ is defined through~\eqref{defS}.
On the other hand, again using $\gtgram = \ptgram - C + \Gamma(C)$,
\begin{equation}\label{thisguy4}
\begin{split}
\trace(V_\perp^\top \gtgram V_\perp) &= \trace(- V_\perp^\top C V_\perp + V_\perp^\top \Gamma(C) V_\perp)
= \trace(-C) + \trace(V_\perp^{} V_\perp^\top \Gamma(C) )\\
&\stackrel{(1)}{\geq} \trace(\Gamma(-C)) + \trace(V_\perp^{} V_\perp^\top \Gamma(C) )
= \trace((I - V_\perp^{} V_\perp^\top) \Gamma(-C) ) \\
&\stackrel{(2)}{=} \trace(V V^\top \Gamma(-C)) = \trace(P).
\end{split}
\end{equation}
Inequality $\stackrel{(1)}{\geq}$ follows from property~\ref{P3}.
Equality $\stackrel{(2)}{=}$ follows from $V V^\top + V_\perp^{} V_\perp^\top = I_n$.
Combining~\eqref{thisguy3} and~\eqref{thisguy4} yields the claim.
\end{proof}

\subsection{Eigenvalue interlacing and the proof of Theorem~\ref{th:sqrt-n}}\label{sectioneigenvalueinterlacing}

This section is devoted to proving Theorem~\ref{th:sqrt-n}.
First, let us summarize what we have already shown (combining Lemmas~\ref{lemmaCNSD}, \ref{derivationdescent} and~\ref{traceP}).
Let $\pt$ be full rank and 2-critical for $\fu$.
For all $u \in \ker(V^\top \gtgram V)$ (a subspace of dimension at least $\dimopt - \dimgt$), we have
\begin{align}\label{2criticalityconsequence}
0 \leq u^\top \Big[- a P + b P^2 + nP^3 \Big]u,
\end{align}
where $P \succeq 0$ with $\trace(P) \leq b$, and
\begin{align*}
  a := \langle S, S\rangle = \sum_{i=1}^{\tildell} \gamma_i^2, && b := \trace(S) = - \sum_{i=1}^{\tildell} \gamma_i.
\end{align*}
Our approach is to show that $M := - a P + b P^2 + n P^3$ restricted to $\ker(V^\top \gtgram V)$ has at least one negative eigenvalue.
We then choose $u$ to be the corresponding eigenvector (which corresponds to a descent direction in $\nabla^2 g(Z)$).
This will contradict inequality~\eqref{2criticalityconsequence}, showing that if $\pt$ is 2-critical then it cannot be full rank.
We then invoke Lemma~\ref{rankdeficiency} to conclude $\pt$ must be a global minimizer of $\fu$.

\begin{lemma}\label{lem:numPevals}
Let $\pt$ be full rank and first-order critical for~\eqref{upstairs}, and assume $\dimopt$ satisfies condition~\eqref{assumptiononk}.
Then either $C = 0$ (and so $g(\pt) = 0$), or the smallest $\dimgt + 1$ eigenvalues of $P = V^\top \Gamma(-C) V \in \PSD(\dimopt)$ are strictly less than $\frac{-b + \sqrt{b^2+4 a n}}{2 n}$.
\end{lemma}
\begin{proof}
Assume $C \neq 0$ (i.e., $b > 0$).
Since the squared $1$-norm in $\reals^{\tildell}$ is at most $\tildell$ times the squared $2$-norm, we have $b^2 \leq \tildell a \leq \dimgt a$.
Combining this bound with assumption~\eqref{assumptiononk} on $\dimopt$ yields
\begin{align}
  \dimopt - \dimgt > \frac{n}{-1/2 + \sqrt{1/4 + n / \dimgt}} \geq \frac{n}{-1/2 + \sqrt{1/4 + n a / b^2}}.
\end{align}
Rearranging gives 
$(\dimopt - \dimgt) \frac{-b + \sqrt{b^2 + 4 a n}}{2 n} > b.$

On the other hand, Lemma~\ref{traceP} implies $\trace(P) \leq b < (\dimopt - \dimgt) \frac{-b + \sqrt{b^2 + 4 a n}}{2 n}.$
Since $\trace(P)$ is at least the sum of the $\dimopt-\dimgt$ largest eigenvalues of $P$, it follows that
$$(\dimopt - \dimgt)\,\lambda_{\dimgt+1}(P)
<
(\dimopt - \dimgt)\,\frac{-b + \sqrt{b^{2}+4an}}{2n},$$
where $\lambda_{\dimgt+1}(P)$ denotes the $(\dimgt+1)$-st smallest eigenvalue of $P$.
Consequently, 
$\lambda_{\dimgt+1}(P) < \frac{-b + \sqrt{b^{2}+4an}}{2n}.$
%
\end{proof}

\paragraph{Sketch of the remaining argument.}
Lemma~\ref{lem:numPevals} tells us that the \(\dimgt + 1\) smallest eigenvalues of \(P\) lie in the interval \([0, \frac{-b + \sqrt{b^2 + 4 a n}}{2n})\).
Intuitively, these eigenvalues should all be strictly positive.
As $P$ is transformed to $-aP + bP^2 + nP^3$, the eigenvalues of $P$ are transformed through the polynomial $t \mapsto -at + bt^2 + nt^3$.
That polynomial is negative on the stated (open) interval (see Figure~\ref{fig:cubicinterlacing}).
Thus, we expect the matrix \(-a P + b P^2 + n P^3\) to have at least \(\dimgt + 1\) negative eigenvalues.
Since \(\ker(V^\top \gtgram V)\) is a subspace of codimension at most \(\dimgt\), the restriction of \(-a P + b P^2 + n P^3\) to this subspace should, by the eigenvalue interlacing theorem (Lemma~\ref{equalityinterlacing1} below), have at least \((\dimgt + 1) - \dimgt = 1\) negative eigenvalues, which would complete the contradiction.
However, to make this argument rigorous, a bit more care is required because we do not know whether the eigenvalues of \(P\) are strictly positive.

\paragraph{Remainder of the proof.}  
Let $\pt$ be full rank and 2-critical for $\fu$.
If $C = 0$, then $g(\pt) = 0$ and $\pt$ is a minimizer; therefore, assume $C \neq 0$.
Let $r$ be the number of nonpositive eigenvalues of the matrix \(M = -a P + b P^2 + n P^3\).
We know that the 1-dimensional function
\begin{align*}
  [0, \infty) \to \reals,\quad \quad t \mapsto - a t + b t^2 + n t^3 = t(-a + b t + n t^2)
\end{align*}
is nonpositive if and only if \(t \in \left[0, \frac{-b + \sqrt{b^2 + 4 a n}}{2n} \right]\), and it vanishes only at \(t = 0\) and \(t = \frac{-b + \sqrt{b^2 + 4 a n}}{2n}\). 
Therefore, by Lemma~\ref{lem:numPevals}, 
$r \geq \dimgt + 1 + q$,
where \(q \geq 0\) denotes the number of eigenvalues of \(P\) equal to the right endpoint of the above interval. See Figure~\ref{fig:cubicinterlacing}.
\begin{figure}[t]
\centering
\includegraphics[width=0.8\textwidth]{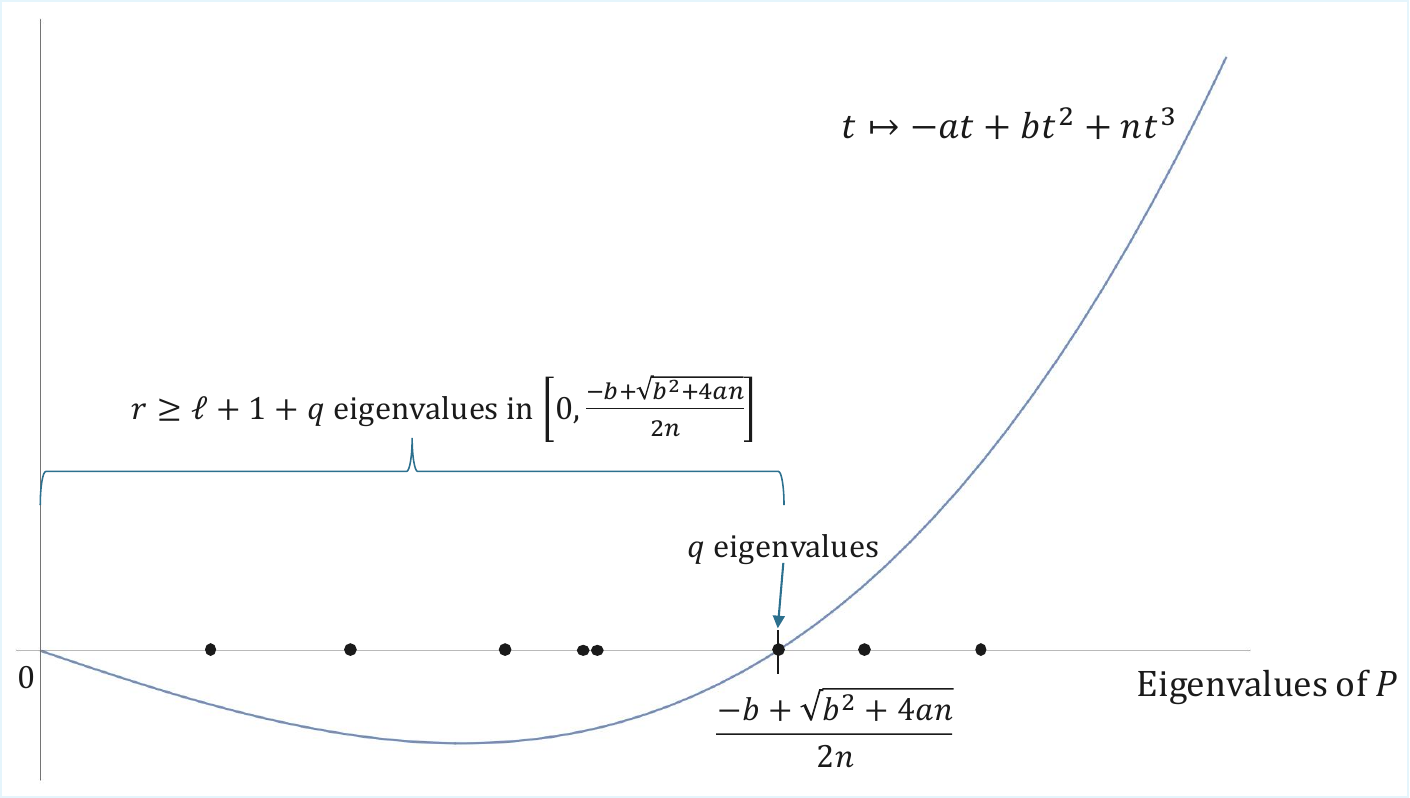}
\caption{The points on the horizontal axis are the eigenvalues of $P$.  The corresponding points on the curve $- a t + b t^2 + n t^3$ are the eigenvalues of $- a P + b P^2 + n P^3$.  See the proof of Theorem~\ref{th:sqrt-n} in Section~\ref{sectioneigenvalueinterlacing}.}
\label{fig:cubicinterlacing}
\end{figure}

Our goal is to determine how many eigenvalues remain nonpositive when $- a P + b P^2 + n P^3$ is restricted to \(\ker(V^\top \gtgram V)\). 
To do so, we apply the classical eigenvalue interlacing theorem, stated below. 
This particular version is a consequence of results in~\citep{horninterlacing1998}; see also Appendix~\ref{app:interlacing}.
\begin{lemma}[Eigenvalue interlacing] \label{equalityinterlacing1}
  Let $M \in \Sym(\dimopt)$ have eigenvalues $m_1 \leq m_2 \leq \cdots \leq m_\dimopt$, and assume $r$ of those eigenvalues are nonpositive: $m_1 \leq m_2 \leq \cdots \leq m_r \leq 0$.
  Let $\calS$ be a $d$-dimensional subspace of $\reals^\dimopt$ with orthonormal basis $S \in \St(\dimopt, d)$.
  Let $\tilde M = S^\top M S$ have eigenvalues $\tilde m_1 \leq \tilde m_2 \leq \cdots \leq \tilde m_d$.\footnote{The matrix $\tilde M$ is called a \emph{compression} of $M$.  It is the restriction of $M$ to $\im(S) = \calS$.}
  If $r-\dimopt+d \geq 1$, then
\begin{align}\label{conclusionequalitiesinterlacing}
\tilde m_1 \leq \tilde m_2 \leq \cdots \leq \tilde m_{r-\dimopt+d} \leq m_r \leq 0.
\end{align}
  If, in addition, $\dim(\calS \cap \ker(M)) < r-\dimopt+d$, then $\tilde m_1 < 0$.
\end{lemma}
Lemma~\ref{equalityinterlacing1} implies that $M = - a P + b P^2 + n P^3$ restricted to $\ker(V^\top \gtgram V)$ has at least
\begin{align}\label{eq:oneplusq}
  r - \dimopt + \dim\!\left(\ker(V^\top \gtgram V)\right) \geq (\dimgt + 1 + q) - \dimopt + \dim\!\left(\ker(V^\top \gtgram V)\right) \geq 1 + q
\end{align}
{nonpositive} eigenvalues, where we used $\dim\!\left(\ker(V^\top \gtgram V)\right) \geq \dimopt - \dimgt$ due to $\rank(\gtgram) \leq \dimgt$.

To complete the argument, it remains to show that at least one of these nonpositive eigenvalues is in fact \emph{negative}---recall from the discussion before Lemma~\ref{lem:numPevals} that we choose \(u \in \ker(V^\top \gtgram V)\) as an eigenvector of $M$ restricted to $\ker(V^\top \gtgram V)$ corresponding to a negative eigenvalue. 

Define the subspace $\mathcal{Q} := \ker(V^\top \gtgram V) \cap \ker(M)$.
The kernel of $M$ is exactly the direct sum of the eigenspaces of $P$ associated to the eigenvalues $0$ and $\frac{-b + \sqrt{b^2+4 a n}}{2 n}$.
Denoting those two eigenspaces of $P$ by $\calP_0, \calP_1$, respectively, we have $\calP_0 + \calP_1 \supseteq \mathcal{Q}$.
In particular, $\mathcal{Q} + \calP_0 \subseteq \calP_0 + \calP_1$.
Using this and $\dim(\calP_1) = q$ (by definition of $q$), we find
\begin{equation}\label{eq:longderivation}
\begin{split}
\dim(\mathcal{Q}\cap\calP_0)
&= \dim(\mathcal{Q})+\dim(\calP_0)-\dim(\mathcal{Q}+\calP_0) \\
&\ge \dim(\mathcal{Q})+\dim(\calP_0)-\dim(\calP_0+\calP_1) \\
&= \dim(\mathcal{Q})+\dim(\calP_0)-(\dim(\calP_0)+\dim(\calP_1)) \\
&= \dim(\mathcal{Q})-q .
\end{split}
\end{equation}

On the other hand, $\mathcal{Q}\cap\calP_0=\{0\}$.
Indeed, if $u\in\mathcal{Q}\cap\calP_0\subseteq\ker(V^\top \gtgram V)$,
then $Pu=0$, and~\eqref{SigmaUperp2} yields
$0=\Sigma u-Pu=\Sigma u.$
This implies $u=0$ since $\Sigma\succ0$.

Since $\mathcal{Q}\cap\calP_0=\{0\}$, inequality~\eqref{eq:longderivation} yields $\dim(\mathcal{Q})\le q$.
Together with~\eqref{eq:oneplusq}, we have
\[
\dim(\mathcal{Q})\le q
< 1+q
\le r-\dimopt+\dim\!\bigl(\ker(V^\top \gtgram V)\bigr).
\]
Therefore, by Lemma~\ref{equalityinterlacing1}, the restriction of $M = - a P + b P^2 + n P^3$ to
$\ker(V^\top \gtgram V)$ has at least one negative eigenvalue
(corresponding to a descent direction in the Hessian $\nabla^2 g(Z)$).
This contradicts the consequence of 2-criticality stated in~\eqref{SigmaUperp}, hence the proof of Theorem~\ref{th:sqrt-n} is complete.

\section{\boldmath Landscape for Gaussian ground truths, $\dimopt \approx \log n$}\label{sec:isotropic}

In this section, we prove a deterministic benign landscape result (stated as Theorem~\ref{thm:gt_dep_determ} below) and we deduce from it a probabilistic benign landscape result (stated as Corollary~\ref{cor:rand_gt} below).
The latter is the precise statement of Theorem~\ref{th:isotropic}.

Recall $\centeringmat = I_n - \frac{1}{n} \ones \ones^\top$ and $z_i^*$ denotes the $i$th row of $\gt$ {as a column}.
Further define $\mu = \frac{1}{n}\sum_{i=1}^n z_i^*$ so that the $i$th row of $\centeringmat \gt$ is $z_i^* - \mu$ {as a column}.
\begin{theorem}
	\label{thm:gt_dep_determ}
	With $\mu = \frac{1}{n}\sum_{i=1}^n z_i^*$, assume
	\begin{equation}
		\label{eq:gt_dep_cond}
		(\dimopt + 2) \sigma_{\min}^2 (\centeringmat \gt) > 4n \max_{i=1, \ldots, n}~\norm{z_i^* - \mu}^2. 
	\end{equation}
	Then any second-order critical configuration $\pt$ of \eqref{upstairs}, with ground truth $\gt$, is a global minimizer.
\end{theorem}
\begin{corollary}[Gaussian ground truth]
	\label{cor:rand_gt}
	Assume the ground truth points (i.e., the rows of $\gt$) are i.i.d.\ Gaussian random vectors with covariance matrix $\Sigma \in \reals^{\dimgt \times \dimgt}$.\footnote{In Section~\ref{sec:isotropic}, $\Sigma$ refers to this covariance; not to be confused with $\Sigma$~\eqref{eq:Vsandstuff} in Section~\ref{sec:sqrt-n}. Likewise, $C$ here refers to a constant, unrelated to the matrix $C$~\eqref{1criticality-and-C} from the previous section.} 
	There exists an absolute constant $C > 0$ such that if
  \begin{align*}
    n \geq C \dimgt && \textrm{ and } && \dimopt \geq C \frac{\trace(\Sigma) + \lambdamax(\Sigma) \log n}{\lambdamin(\Sigma)},
  \end{align*}
	then with probability at least $1 - \frac{C}{n^2}$ (over the distribution of $\gt$), any second-order critical point $\pt$ of~\eqref{upstairs}, with ground truth $\gt$, is a global minimizer.
\end{corollary}
The power of $n$ in the probability bound can be arbitrary; this only affects the constants. The $\log n$ term comes from upper bounding the maximum norm of $n$ independent Gaussian vectors.
It will be clear from the proof that the Gaussian assumption can replaced by sub-Gaussian with suitable adaptations.
\begin{proof}[Proof of Corollary~\ref{cor:rand_gt}.]
	We bound the terms in~\eqref{eq:gt_dep_cond} so that we can apply Theorem~\ref{thm:gt_dep_determ}.
Note that
\begin{align}\label{eq:usefulnumber3}
\max_{i=1, \ldots, n}~\norm{z_i^* - \mu}
\leq \max_{i=1, \ldots, n} \norm{z_i^*} + \|\mu\|
		\leq 2 \max_{i=1, \ldots, n} \norm{z_i^*},
\end{align}
because $\|\mu\| \leq \frac{1}{n} \sum_i \|z_i^*\| \leq \max_{i} \norm{z_i^*}$.
	
	Now consider the random distribution of the matrix $\gt$.
	We can write $\gt = W \Sigma^{1/2}$,
	where $W \in \reals^{n \times \dimgt}$ is a matrix whose entries are i.i.d.\ standard normal random variables.
	In what follows, $C > 0$ denotes a sufficiently large absolute constant that may change from one usage to the next, and we always assume $n \geq C \dimgt$ (i.e., $n$ is big enough).
	
	First, we bound the right-hand side of~\eqref{eq:usefulnumber3}.
	We have $z_i^* = \Sigma^{1/2} w_i$, where $w_1, \dots, w_n$ in $\reals^\dimgt$ are i.i.d.\ random vectors with i.i.d.\ standard normal entries.
	The Hanson--Wright inequality~\citep[Thm.~6.3.2]{vershynin_2018} implies
	$$\mathbb{P}\Big(\|z_i^*\| - \trace(\Sigma)^{1/2} > t \Big) \leq 2 \exp\!\Big(-\frac{t^2}{C \lambda_{\max}(\Sigma)}\Big), \quad \quad \forall t \geq 0.$$
Using the identity $(a+b)^2 \leq 2(a^2+b^2)$, and setting $t^2 = C \lambda_{\max}(\Sigma) \log(n^3)$, we have
	$$\mathbb{P}\Big(\|z_i^*\|^2 > 2 (t^2 + \trace(\Sigma)) \Big) \leq 2 \exp\!\Big(-\frac{t^2}{C \lambda_{\max}(\Sigma)}\Big) = \frac{2}{n^3}.$$
	Hence, a union bound implies
	\begin{align}\label{eq:usefulnumber31}
		\max_{i=1, \ldots, n} \|z_i^*\|^2 \leq C \big(\trace(\Sigma) + \lambda_{\max}(\Sigma) \log n\big)
	\end{align}
	with probability at least $1 - \frac{C}{n^2}$. 
		
	Now consider the term $\sigma_{\min}^2 (\centeringmat \gt)$ on the left-hand side of~\eqref{eq:gt_dep_cond}.
	Note that
	\begin{align*}
		\sigmamin(\centeringmat \gt)
		&= \sigmamin (\centeringmat W \Sigma^{1/2}) 
		\geq \lambdamin^{1/2}(\Sigma) \, \sigmamin(\centeringmat W).
	\end{align*}
	Furthermore,
	\begin{align*}
		\sigmamin(\centeringmat W)
		&= \sigmamin(W - n^{-1} \onevec \onevec^\top W) 
		\geq \sigmamin(W) - \frac{1}{\sqrt{n}} \norm{W^\top \onevec}.
	\end{align*}
	Standard matrix concentration inequalities \citep[for example,][Thm.~4.6.1]{vershynin_2018} imply that $\sigmamin(W) \geq \frac{\sqrt{n}}{2}$ with probability at least $1 - \frac{C}{n^2}$, because $n \geq C \dimgt$.
	Since $W^\top \ones$ is a Gaussian vector in $\reals^\dimgt$ with mean zero and covariance $n I_\dimgt$, similar reasoning as above implies
	$$\norm{W^\top \onevec}^2 \leq C n(\dimgt + \log n)$$
	with probability at least $1 - \frac{C}{n^2}$.
	In those events, we obtain 
	\begin{equation}\label{eq:usefulnumber32}
	\begin{split}
		\sigmamin(\centeringmat \gt)
		& \geq \lambdamin^{1/2}(\Sigma) \left( \frac{\sqrt{n}}{2} -  \sqrt{C(\dimgt + \log n) } \right) \\
		&\geq \frac{\sqrt{n}}{4}\lambdamin^{1/2}(\Sigma).
	\end{split}
	\end{equation}
	Plugging the bounds~\eqref{eq:usefulnumber3},~\eqref{eq:usefulnumber31} and~\eqref{eq:usefulnumber32} into Theorem~\ref{thm:gt_dep_determ} and using a union bound on the failure probabilities,
	we obtain the result.
\end{proof}

We now turn to the proof of Theorem~\ref{thm:gt_dep_determ}. 
\begin{proof}[Proof of Theorem~\ref{thm:gt_dep_determ}]
	Let $\pt$ be a 2-critical point of \eqref{upstairs}.	
	Without loss of generality, we can assume $\pt$ is centered.
	We can also assume $\gt$ is centered---if not, replace $\gt$ by $\centeringmat \gt$ in the following.
	
	A rearrangement of the inequality \eqref{eq:cute_1} with $T = I_{\dimopt}$ in Lemma~\ref{lem:cute} yields, for any $n \times n$ positive semidefinite matrix $S$,
\begin{align}\label{eq:andrewconsequenceof2crit}
		(\dimopt + 2) \innersmall{S}{(\Delta^* \circ \Delta)(\gt \gt\transpose - \pt \pt\transpose)} \leq  2 \innersmall{(\Delta^* \circ \Delta) (\gt \gt^\top)}{S}.
\end{align}
	We now choose $S \coloneqq (\gt - \pt R^\top)(\gt - \pt R^\top)\transpose$ for some $R \in \reals^{\dimgt \times \dimopt}$ to be specified later.
	In the following, we bound the quantities on the left- and right-hand sides of~\eqref{eq:andrewconsequenceof2crit}.
	
\paragraph{Bounding the right-hand side.} We can upper bound the quantity on the right-hand side of~\eqref{eq:andrewconsequenceof2crit} as
	\begin{align*}
		\innersmall{(\Delta^* \circ \Delta) (\gt \gt^\top)}{S}
		&\leq \opnormsmall{(\Delta^* \circ \Delta) (\gt \gt^\top)} \trace(S) \\
		&= \opnormsmall{(\Delta^* \circ \Delta) (\gt \gt^\top)} \sqfrobnormsmall{\gt - \pt R^\top}.
	\end{align*}
	By the formula~\eqref{eq:deltastardelta},
	we have
	\begin{align*}
		\opnormsmall{(\Delta^* \circ \Delta) (\gt \gt^\top)}
		&\leq \opnormsmall{\gt \gt^\top} + \frac{n}{2} \opnormsmall{\centeringmat \Diag(\gt \gt^\top) \centeringmat}
		+ \frac{1}{2} \opnormsmall{\trace(\gt \gt^\top) \centeringmat} \\
		&\leq \opnormsmall{\gt}^2 + \frac{n}{2} \max_{i=1, \ldots, n}~\norm{z_i^*}^2 + \frac{1}{2} \frobnormsmall{ \gt}^2.
	\end{align*}
	Using that $\opnormsmall{\gt}^2 \leq \frobnormsmall{\gt}^2
		\leq n \max_{i=1, \ldots, n}~\norm{z_i^*}^2$, we conclude
$$\opnormsmall{(\Delta^* \circ \Delta) (\gt \gt^\top)} \leq 2 n \max_{i=1, \ldots, n}~\norm{z_i^*}^2.$$
	
\paragraph{Bounding the left-hand side.} By expanding $S$ and using the 1-criticality condition $(\Delta^* \circ \Delta)(\gt \gt\transpose - \pt \pt\transpose) \pt = 0$ (see~\eqref{eq:1crit}), the quantity on the left-hand side of~\eqref{eq:andrewconsequenceof2crit} becomes
	\begin{align}
		\innersmall{S}{(\Delta^* \circ \Delta)(\gt \gt\transpose - \pt \pt\transpose)}
		&= \innersmall{\gt \gt\transpose - \pt \pt\transpose}{(\Delta^* \circ \Delta)(\gt \gt\transpose - \pt \pt\transpose)} \nonumber \\
		&= \sqfrobnormsmall{\Delta(\gt \gt\transpose - \pt \pt\transpose)} 
		\stackrel{(1)}{\geq} \sqfrobnormsmall{\gt \gt\transpose - \pt \pt\transpose}. \label{eq:forlaststep}
	\end{align}
	Inequality $\stackrel{(1)}{\geq}$ follows from inequality~\eqref{eq:lowerboundonDeltacircDelta}.
	Now, choose $R \in \reals^{\dimgt \times \dimopt}$ to minimize $\sqfrobnormsmall{\gt - \pt R^\top}$, i.e., such that $\pt\transpose (\gt - \pt R^\top) = 0$.
	Let $P_{\im(\pt)}^\perp$ be the orthogonal projection matrix onto the orthogonal complement of the image of $\pt$.
	Observe $P_{\im(\pt)}^\perp \pt = 0$ and $P_{\im(\pt)}^\perp \gt = \gt - \pt R^\top$.
	Therefore,
	\begin{align*}
		\sqfrobnormsmall{\gt \gt\transpose - \pt \pt\transpose}
		&\geq \sqfrobnormsmall{(\gt \gt\transpose - \pt \pt\transpose) P_{\im(\pt)}^\perp} = \sqfrobnormsmall{\gt \gt\transpose P_{\im(\pt)}^\perp} \\
		&= \sqfrobnormsmall{\gt (\gt - \pt R^\top)\transpose}
		\geq \sigma_{\min}^2(\gt) \sqfrobnormsmall{\gt - \pt R^\top}.
	\end{align*}
	
\paragraph{Wrapping up.} Combining the above, inequality~\eqref{eq:andrewconsequenceof2crit} becomes
	\[
		(\dimopt + 2) \sigma_{\min}^2(\gt) \sqfrobnormsmall{\gt - \pt R^\top} \leq 4 n \Big(\max_{i=1, \ldots, n}~\norm{z_i^*}^2\Big) \sqfrobnormsmall{\gt - \pt R^\top}.
	\]
	Under condition~\eqref{eq:gt_dep_cond}, and recalling that both $\pt$ and $\gt$ are centered (i.e., $\centeringmat \gt = \gt$ and $\mu = 0$), we must have $\sqfrobnormsmall{\gt - \pt R^\top} = 0$. 
	It follows that $S = 0$, and hence inequality~\eqref{eq:forlaststep} reveals that $\sqfrobnormsmall{\gt \gt\transpose - \pt \pt\transpose} = 0$ also.
\end{proof}

\begin{remark}[Towards incomplete graphs]
Recall that Lemma~\ref{lem:cute} holds for incomplete graphs (with $\Delta$ suitably modified).
Therefore, the proof of Theorem~\ref{thm:gt_dep_determ} uses that the graph is complete in only two places: (i) to establish the rigidity-like identity~\eqref{eq:lowerboundonDeltacircDelta}, and (ii) to establish an upper bound on $\opnormsmall{(\Delta^* \circ \Delta) (\gt \gt^\top)}$.
\end{remark}

\section{Examples of spurious configurations}\label{sec:counterexamples}

This section exhibits \emph{simple} counterexamples showing that, without relaxation,~\eqref{eq:snl} can admit strict spurious second-order critical (or 2-critical) points even when the graph is complete.
Here, ``strict'' means that the Hessian is positive definite modulo translation and rotation symmetries.
In particular, strict 2-critical points are local minimizers.
Moreover, they are stable under sufficiently small perturbations (a consequence of the implicit function theorem).
This implies that the set of ground truth configurations for which~\eqref{eq:snl} admits a spurious local minimizer has positive measure.

Concretely, the objective $g$ is invariant under both translations, $\pt \mapsto \pt + \ones v^\top$ for any $v \in \reals^\dimopt$, and rotations, $\pt \mapsto \pt Q$ for any orthogonal matrix $Q \in \mathrm{O}(\dimopt)$.
As a result, if $\pt$ is a full-rank 2-critical configuration, then the kernel of $\nabla^2 g$ contains the following two subspaces, which can also be verified directly from equation~\eqref{eq:derivatives}:\footnote{We recall that equation~\eqref{eq:derivatives} remains valid when the operator $\Delta$ is interpreted as a linear map $\Delta \colon \mathrm{Sym}(n)\to \mathrm{Cent}(n)$, as specified in definition~\eqref{EDMmap}.}
\begin{equation}\label{eq:symmetrysubspaces}
\begin{split}
\textbf{(translations)} &\qquad\qquad \big\{\ones v^\top : v \in \reals^\dimopt\big\}, \qquad\quad \text{and} \\
\textbf{(rotations)} &\qquad\qquad \big\{\pt \Omega : \Omega \in \Skew(\dimopt)\big\}.
\end{split}
\end{equation}
These subspaces have dimensions $\dimopt$ and $\dimopt(\dimopt-1)/2$, respectively.
Since they intersect only at zero, their sum is a subspace of dimension $\dimopt + \dimopt(\dimopt-1)/2 = \dimopt(\dimopt+1)/2$ contained in the kernel of $\nabla^2 g$.
Therefore, to verify that a full-rank 2-critical configuration $\pt$ is \emph{strict}, it suffices to show the kernel of $\nabla^2 \fu$ has dimension exactly $\dimopt(\dimopt+1)/2$.

In Section~\ref{sec:generalconstruction}, we present a general construction that produces spurious strict 2-critical points, along with a supporting proposition.
In Section~\ref{sec:explicitexamples}, we apply this proposition to construct explicit configurations that are strict spurious 2-critical points.

\subsection{A general construction}\label{sec:generalconstruction}

Our examples {all} follow the same general construction.
Let $\dimopt = \dimgt$ and $n \geq \dimgt + 2$.
Choose an affine subspace $\mathcal{L} \subseteq \reals^\dimgt$ of dimension $\dimgt - 1$, and consider (for now) \emph{any} set of $n-2$ points $z_1^*, \ldots, z_{n-2}^*$ in $\mathcal{L}$.
Let $z_{n-1}^* \in \reals^\dimgt$ be any point not contained in $\mathcal{L}$.
Define $z_n^*$ to be the \emph{reflection} of $z_{n-1}^*$ across $\mathcal{L}$, i.e., the reflection fixes $\mathcal{L}$.
Let $\gt \in \reals^{n \times \dimgt}$ be the matrix whose rows are $z_1^*, \ldots, z_n^*$.
Lastly, define a configuration $\pt \in \reals^{n \times \dimgt}$ which is identical to $\gt$, except its final row is $z_{n} = z_{n-1}^*$.
In particular, the last two rows of $\pt$ are identical.
Figure~\ref{fig:counterexample} gives an example of configurations $\gt$ and $\pt$.

\begin{proposition}\label{generalcounterexample}
  Consider the configurations $\gt$ and $\pt$ defined above.
  Then $\pt$ is a spurious first-order critical point for~\eqref{eq:snl} with ground truth $\gt$, $\dimopt = \dimgt$, and the complete graph.
  Moreover, if
  \begin{align}\label{conditioncounterexample}
    \sum_{i=1}^{n-2} (z_i^* - z_{n-1}^*) (z_i^* - z_{n-1}^*)^\top \succeq \|z_{n-1}^* - z_{n}^*\|^2 I_\dimgt
  \end{align}
  then $\pt$ is a spurious second-order critical point.
  If inequality~\eqref{conditioncounterexample} is strict (i.e., $\succ$), then $\pt$ is a strict spurious second-order critical point.
\end{proposition}

\begin{proof}
\textbf{(1-criticality)} By construction, $\pt$ and $\gt$ differ only in their last point, so the pairwise distances corresponding to $\pt$ and $\gt$ differ only between the last two points.
Therefore, all entries of $\Delta(\pt \pt^\top - \gt^{} \gt^\top)$ are zero except for the ones at indices $(n - 1, n)$ and $(n, n - 1)$, which are equal to $\frac{1}{2}(\|z_{n-1} - z_{n}\|^2 - \|z_{n-1}^* - z_{n}^*\|^2) = -\frac{1}{2} \|z_{n-1}^* - z_{n}^*\|^2$.
For brevity, define $\alpha := \|z_{n-1}^* - z_{n}^*\|^2$.

By~\eqref{eq:adjoint}, all entries of $(\Delta^* \circ \Delta)(\pt \pt^\top - \gt \gt^\top)$ are zero except for the bottom right $2 \times 2$ block, which equals $\frac{\alpha}{2} \left( \begin{smallmatrix} -1 & \phantom{-}1 \\ \phantom{-}1 & -1 \end{smallmatrix}\right)$.
Using~\eqref{eq:derivatives} and that the last two rows of $\pt$ are identical, we conclude that $\nabla \fu(\pt) = (\Delta^* \circ \Delta)(\pt\pt^\top - \gt\gt^\top) \pt$ equals zero.
Therefore, $\pt$ is first-order critical.
\\
\\
\noindent \textbf{(2-criticality)} Reusing the above computations, we find that
\begin{equation}\label{eq:Hess1}
  \begin{split}
    \big\langle \dot \pt {\dot \pt}^\top,  (\Delta^* \circ \Delta)(\ptgram - \gtgram) \big\rangle
    &=
      \frac{\alpha}{2} \left \langle
        \begin{bmatrix} \|{\dot z}_{n-1}\|^2 & \langle {\dot z}_{n-1}, {\dot z}_{n} \rangle \\ \langle {\dot z}_{n-1}, {\dot z}_{n} \rangle & \|{\dot z}_{n}\|^2 \end{bmatrix},
        \begin{bmatrix} -1 & \phantom{-}1 \\ \phantom{-}1 & -1 \end{bmatrix}
      \right \rangle \\
    &= - \frac{\alpha}{2} \|{\dot z}_{n-1} - {\dot z}_{n}\|^2,
  \end{split}
\end{equation}
for all $\dot \pt \in \reals^{n \times \dimgt}$.
Additionally, using the identity~\eqref{eq:identityinnerproduct} and that the last two rows of $\pt$ both equal $z_{n-1}^*$,
\begin{equation}\label{eq:Hess2}
  \begin{split}
    \frac{1}{2} \sqfrobnormsmall{\Delta(\dot \pt \pt^\top + \pt \dot \pt^\top)}
    &= \frac{1}{2} \sum_{i,j = 1}^{n} \langle {\dot z}_i - {\dot z}_j, z_i - z_j\rangle^2 \\
    &= \frac{1}{2} \sum_{i,j = 1}^{n-2} \langle {\dot z}_i - {\dot z}_j, z_i^* - z_j^* \rangle^2
      + \sum_{i=1}^{n-2} \sum_{j=n-1}^{n} \langle {\dot z}_i - {\dot z}_j, z_i^* - z_{n-1}^*\rangle^2.
  \end{split}
\end{equation}
For brevity, define $w_i := z_i^* - z_{n-1}^*$.
Using the identity $(a-b)^2 + (a-c)^2 = \frac{1}{2} (b-c)^2 + \frac{1}{2} (b+c-2a)^2$ for $a,b,c \in \reals$, we can rewrite the second term on the right-hand side:
\begin{equation}\label{eq:Hess3}
  \begin{split}
    \sum_{i=1}^{n-2} \sum_{j=n-1}^{n} \langle {\dot z}_i - {\dot z}_j, w_i\rangle^2
    &=
      \sum_{i=1}^{n-2} \langle {\dot z}_i - {\dot z}_{n-1}, w_i\rangle^2 + \langle {\dot z}_i - {\dot z}_{n}, w_i\rangle^2
    \\
    &=
      \frac{1}{2} \sum_{i=1}^{n-2} \langle {\dot z}_{n-1} - {\dot z}_{n}, w_i \rangle^2 + \langle {\dot z}_{n-1} + {\dot z}_{n} - 2 {\dot z}_i, w_i \rangle^2.
  \end{split}
\end{equation}
Using the expression~\eqref{eq:derivatives} for $\nabla^2 g$, and combining equations~\eqref{eq:Hess1},~\eqref{eq:Hess2} and~\eqref{eq:Hess3},
\begin{align*}
  \langle \dot \pt, \nabla^2 \fu(\pt)[\dot \pt]\rangle
  =&
     \big\langle \dot \pt {\dot \pt}^\top,  (\Delta^* \circ \Delta)(\ptgram - \gtgram) \big\rangle
     + \frac{1}{2} \sqfrobnormsmall{\Delta(\dot \pt \pt^\top + \pt \dot \pt^\top)} \\
  =&
     \frac{1}{2} ({\dot z}_{n-1} - {\dot z}_{n})^\top\bigg( -\alpha I_\dimgt + \sum_{i=1}^{n-2} w_i w_i^\top \bigg) ({\dot z}_{n-1} - {\dot z}_{n})
  \\
   &+ \frac{1}{2} \sum_{i,j = 1}^{n-2} \langle {\dot z}_i - {\dot z}_j, z_i^* - z_j^* \rangle^2
     + \frac{1}{2} \sum_{i=1}^{n-2} \langle {\dot z}_{n-1} + {\dot z}_{n} - 2 {\dot z}_i, w_i \rangle^2.
\end{align*}
Under assumption~\eqref{conditioncounterexample}, we conclude $\langle \dot \pt, \nabla^2 \fu(\pt)[\dot \pt]\rangle \geq 0$, and so $\pt$ is a 2-critical point.

Assuming~\eqref{conditioncounterexample} holds strictly, it remains to show that $\pt$ is a \emph{strict} 2-critical configuration.
It is enough to show that if $\dot \pt$ satisfies $\langle \dot \pt, \nabla^2 \fu(\pt)[\dot \pt]\rangle = 0$, then it lies in the sum of the subspaces~\eqref{eq:symmetrysubspaces}.
There exists $c \in (0, 1]$ such that $(1 - c) \sum_{i=1}^{n-2} w_i w_i^\top \succeq \alpha I_\dimgt$.\footnote{Indeed, if $A \succ B \succeq 0$, then define $c = \lambda_{\min}(A-B) / \lambda_{\max}(A)$.}
It follows that
\begin{align*}
  ({\dot z}_{n-1} - {\dot z}_{n})^\top\bigg( -\alpha I_\dimgt + \sum_{i=1}^{n-2} w_i w_i^\top \bigg) ({\dot z}_{n-1} - {\dot z}_{n}) \geq c \sum_{i=1}^{n-2} \langle {\dot z}_{n-1} - {\dot z}_{n}, w_i \rangle^2.
\end{align*}
Combining this with our work above, we deduce
\begin{align*}
  \langle \dot \pt, \nabla^2 \fu(\pt)[\dot \pt]\rangle
  \geq
  \frac{c}{2} \sqfrobnormsmall{\Delta(\dot \pt \pt^\top + \pt \dot \pt^\top)} && \text{for all $\dot \pt \in \reals^{n \times \dimgt}$.}
\end{align*}
Therefore, if $\dot \pt$ satisfies $\langle \dot \pt, \nabla^2 \fu(\pt)[\dot \pt]\rangle = 0$, then $\Delta(\dot \pt \pt^\top + \pt \dot \pt^\top) = 0$.
By invertibility of $\Delta \colon \Cent(n) \to \Holl(n)$, this implies $\centeringmat (\dot \pt \pt^\top + \pt \dot \pt^\top) \centeringmat = 0$.
In other words, defining $\Xi := \centeringmat \dot \pt \pt^\top \centeringmat$, we have $\Xi = \centeringmat \dot \pt \pt^\top \centeringmat \in \Skew(n)$.

Define $\pt_c = \centeringmat \pt$, and observe assumption~\eqref{conditioncounterexample} implies $\rank(\pt_c) = \dimgt$.
We then have $\centeringmat \dot \pt = \Xi \pt_c^{} (\pt_c^\top \pt_c^{})^{-1}$, i.e., $\dot \pt^\top \centeringmat = - (\pt_c^\top \pt_c^{})^{-1} \pt_c^\top \Xi $, which gives
$$\centeringmat^{} \dot \pt = (\centeringmat^{} \dot \pt \pt_c^\top) \pt_c^{} (\pt_c^\top \pt_c^{})^{-1}  = - (\pt_c^{} \dot \pt^\top \centeringmat^{}) \pt_c^{} (\pt_c^\top \pt_c^{})^{-1} = \pt_c^{} (\pt_c^\top \pt_c)^{-1} \pt_c^\top \Xi \pt_c^{} (\pt_c^\top \pt_c^{})^{-1}.$$
We conclude that $\centeringmat \dot \pt = \centeringmat \pt \Omega$ for some $\Omega \in \Skew(\dimgt)$, which implies that $\dot \pt$ lies in the sum of the subspaces~\eqref{eq:symmetrysubspaces} with $\dimopt = \dimgt$.
\end{proof}

\subsection{Explicit examples}\label{sec:explicitexamples}

We start off with a counterexample in the plane.
Then we turn to higher dimensions, where we establish the minimal number of points necessary for spurious strict 2-critical configurations to exist.

\subsection*{In the plane}\label{subsec:counter-example-2d}

Let $n = 7$, $\dimgt = \dimopt = 2$, and consider the ground truth configuration $\gt$ and the configuration $\pt$ defined as follows:
\begin{align}\label{configurationplaneex}
  \gt =
  \begin{bmatrix}
     -2 & -1 & 0 & 1 & 2 & 0 & 0\\
     0 & 0 & 0 & 0 & 0 & 1 & -1
  \end{bmatrix}^\top
  && \text{and} &&
  \pt =
  \begin{bmatrix}
    -2 & -1 & 0 & 1 & 2 & 0 & 0 \\
    0 & 0 & 0 & 0 & 0 & 1 & 1
  \end{bmatrix}^\top.
\end{align}
Figure~\ref{fig:counterexample} represents these configurations.
\begin{proposition}\label{counterexampleplane}
  The configuration $\pt$ given by~\eqref{configurationplaneex} is a strict spurious second-order critical point for~\eqref{eq:snl} with ground truth $\gt$, $\dimopt = \dimgt$, and the complete graph.
\end{proposition}
\begin{proof}
  We appeal to Proposition~\ref{generalcounterexample}.
  It is enough to verify condition~\eqref{conditioncounterexample} holds strictly:
  \begin{align*}
    \sum_{i=1}^{n-2} (z_i^* - z_{n-1}^*) (z_i^* - z_{n-1}^*)^\top
    = \begin{bmatrix} 10 & 0 \\ 0 & 5 \end{bmatrix}
    \succ
    4 I_2 =
    \|z_{n-1}^* - z_{n}^*\|^2 I_2. && \qedhere
  \end{align*}
\end{proof}

A simple way to generalize the construction above to any $n \geq 7$ is to add points along the horizontal axis.
The same construction as above---excluding the point at the origin---yields a \emph{non-strict} spurious second-order critical point (with $n = 6$).
For $n \leq 5$, this type of construction does not yield spurious second-order critical configurations.

\subsection*{In higher dimensions: fewest points for non-benign landscape}

Given a dimension $\dimgt$, what is the minimum number of points $n$ for which~\eqref{eq:snl}, with the complete graph, can admit a strict spurious second-order critical point?
As shown in Section~\ref{sec:lowrank}, $n$ must be at least $\dimgt + 2$.
We now show that for dimensions $\dimgt \geq 5$, strict spurious second-order critical points do exist when $n = \dimgt + 2$.

Define the ground truth configuration $\gt \in \reals^{(\dimgt + 2) \times \dimgt}$ as the set of standard basis vectors $e_1, \ldots, e_\dimgt \in \reals^\dimgt$, the origin, and the point $\frac{2}{\dimgt} \ones$ (the reflection of the origin across the standard simplex).
Define the configuration $\pt \in \reals^{(\dimgt + 2) \times \dimgt}$ as the same standard basis vectors together with two copies of the origin.
Geometrically, the first $\dimgt$ points are the vertices of the standard simplex, and the remaining two points lie on the line orthogonal to the simplex and passing through its centroid.

\begin{proposition}\label{originalcounterexample}
  If $\dimgt \geq 5$, then $\pt$ as just given is a strict spurious second-order critical point for~\eqref{eq:snl} with ground truth $\gt$, $\dimopt = \dimgt$, and the complete graph.
  If $\dimgt =4$, it is a non-strict spurious second-order critical point.
\end{proposition}
\begin{proof}
  We appeal to Proposition~\ref{generalcounterexample}.
  It is enough to verify condition~\eqref{conditioncounterexample}:
  \begin{align*}
    \sum_{i=1}^{n-2} (z_i^* - z_{n-1}^*) (z_i^* - z_{n-1}^*)^\top
    = \sum_{i=1}^{\dimgt} e_i e_i^\top = I_\dimgt
    \succeq
    \frac{4}{\dimgt} I_\dimgt =
    \|z_{n-1}^* - z_{n}^*\|^2 I_\dimgt,
  \end{align*}
  provided $\dimgt \geq 4$.
  If $\dimgt \geq 5$, then the inequality is strict.
\end{proof}

For dimensions $\dimgt \in \{3, 4\}$ and $n=7$ points, consider a similar construction where the ground truth consists of the $\dimgt$ standard basis vectors, $5-\dimgt$ additional points placed at the centroid of the simplex $\frac{1}{\dimgt} \ones$, one point located at $\frac{1}{2\dimgt} \ones$, and the final point located at $\frac{3}{2\dimgt} \ones$ (the reflection of $\frac{1}{2\dimgt} \ones$ across the simplex).
Define $\pt$ as the same set of points, except with $\frac{3}{2\dimgt} \ones$ replaced by $\frac{1}{2\dimgt} \ones$.
One checks (e.g., using Proposition~\ref{generalcounterexample}) that the corresponding configuration with $n=7$ points is strict second-order critical.

Combining this observation with the counterexamples from Propositions~\ref{counterexampleplane} and~\ref{originalcounterexample}, we conclude that strict spurious second-order critical points exist for $n = \max\{\dimgt + 2, 7\}$ points when $\dimgt \geq 2$.
Similar arguments show that (non-strict) second-order critical points exist for $n = \max\{\dimgt + 2, 6\}$.
We suspect these bounds are \emph{tight}, in the sense that if $\dimgt \leq 4$ and $n=6$, the landscape of~\eqref{eq:snl} with $\dimopt = \dimgt$ and the complete graph is benign for all ground truths, except for a set of measure zero.

\section{Numerical observations}\label{sec:numerics}

In this section, we explore the landscape of~\eqref{eq:snl} through numerical simulations.
How common are spurious local minimizers, and how effective is rank relaxation?
To answer this question empirically, we proceed as follows.
First, we generate a ground truth configuration at random ($n$ independent random standard Gaussian points in dimension $\dimgt$), which defines an instance of the SNL optimization problem with noiseless measurements.
Then, for a given optimization rank $\dimopt \geq \dimgt$, we run a trust-region algorithm initialized at a random initial configuration ($n$ independent random standard Gaussian points in dimension $\dimopt$), and we observe where it converges.
We repeat this process multiple times to collect statistics.

Even though the theoretical results in this paper apply specifically to the complete graph case, we are also interested in exploring numerically how the landscape changes for sparser graphs.
For this reason, we perform experiments with random \erdosrenyi{} graphs of various densities.

In all plots, the black curve indicates the probability that the random \erdosrenyi{} graph is connected at the given density.
This serves as a benchmark, since the ground truth cannot be recovered when the graph is disconnected.

The code to reproduce these numerical experiments is available in a public Github repository.\footnote{\url{https://github.com/qrebjock/snl-experiments}}

\subsection*{Recovering the ground truth}

In this first set of experiments, we define ``success'' as the ability to recover the ground truth.
Specifically, we draw a ground truth configuration $\gt \in \reals^{n \times \dimgt}$ at random, let the algorithm converge to $\pt \in \reals^{n \times \dimopt}$, and say that the algorithm succeeds if, after projection to dimension $\dimgt$, the points $\pt$ match $\gt$ up to rigid motion symmetries, within a tolerance of $10^{-10} \sqrt{n\dimgt}$ on the norm of the difference.
Here, by projection, we mean forming the best rank-$\dimgt$ approximation of the Gram matrix $\ptgram$ and reconstructing the corresponding configuration.
Figure~\ref{fig:mds-norms} shows the results of this experiment.

\begin{figure}[p]
  \centering
  \begin{subfigure}[t]{\textwidth}
    \centering
    \includegraphics[width=0.495\textwidth]{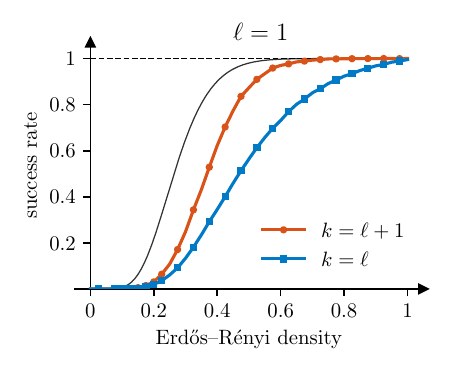}
    \includegraphics[width=0.495\textwidth]{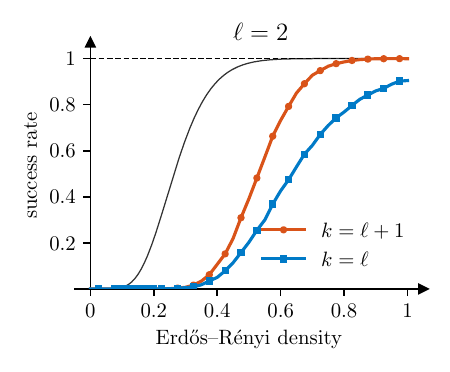}
    \caption{Simulations with $n = 10$ points.}
  \end{subfigure}
  \vspace{1em}
  \begin{subfigure}[t]{\textwidth}
    \centering
    \includegraphics[width=0.495\textwidth]{figures/l=1_n=50_norm.pdf}
    \includegraphics[width=0.495\textwidth]{figures/l=2_n=50_norm.pdf}
    \caption{Simulations with $n = 50$ points.}
  \end{subfigure}
  \caption{Success rates for ground truth recovery.
    The blue curves never attain a success rate of 1.}\label{fig:mds-norms}
\end{figure}

\subsection*{Driving the cost to zero}

In this second set of experiments, we define ``success'' as the ability to drive the cost to zero.
Specifically, we let the algorithm converge to $\pt \in \reals^{n \times \dimopt}$, and say that it succeeds if $\fu(\pt) \leq 10^{-10}$.
Note that this is a weaker success criterion than in the previous section: while the ground truth configuration has zero cost, other configurations may also achieve zero cost, particularly in the case of sparse graphs.
Figure~\ref{fig:mds-costs} shows the results of this experiment.

\begin{figure}[p]
  \centering

  \begin{subfigure}[t]{\textwidth}
    \centering
    \includegraphics[width=0.49\textwidth]{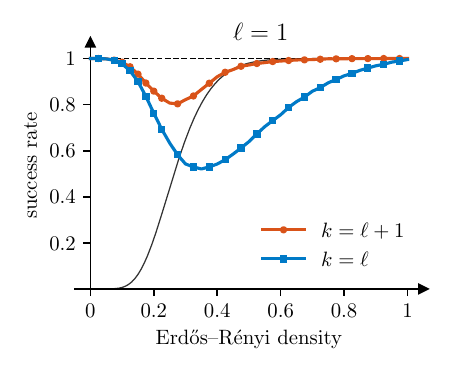}
    \hfill
    \includegraphics[width=0.49\textwidth]{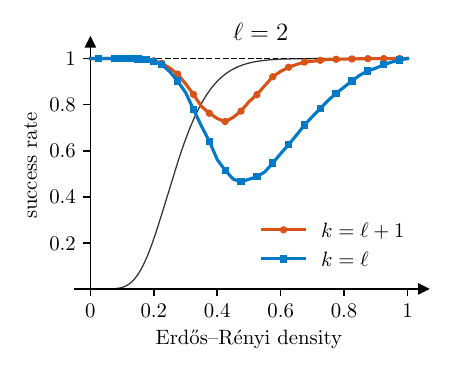}
    \caption{Simulations with $n = 10$ points.}
  \end{subfigure}

  \vspace{1em}

  \begin{subfigure}[t]{\textwidth}
    \centering
    \includegraphics[width=0.49\textwidth]{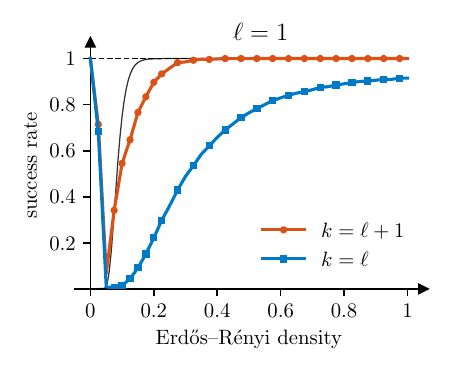}
    \hfill
    \includegraphics[width=0.49\textwidth]{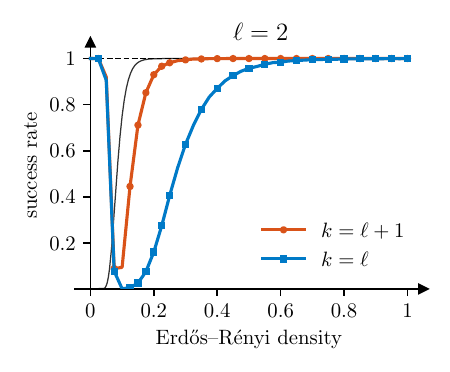}
    \caption{Simulations with $n = 50$ points.}
  \end{subfigure}

  \caption{Success rates for global optimality.
    The blue curves never attain a success rate of 1, except when the density is zero.}\label{fig:mds-costs}
\end{figure}

\subsection*{A few comments}

Consider the complete graph case without rank relaxation.
The frequency of spurious configurations depends strongly on the ground truth dimension $\dimgt$ and the number of points $n$.

For example, with $n = 50$, spurious second-order critical points are common for $\dimgt = 1$, but rare for $\dimgt \geq 2$, as shown in Figure~\ref{fig:mds-norms}.
The success rate is approximately 92\% for $\dimgt = 1$, and increases to 99.98\% for $\dimgt = 2$.
This may be due to the fact that in dimension $\dimgt = 1$, SNL resembles a discrete problem of recovering signs.

Additionally, for a given dimension $\dimgt$, spurious second-order critical points tend to be rare when $n$ is either very small or very large, and are most frequent in the intermediate regime.

The numerical experiments of this section focus on the noiseless case.
The results are stable to moderate amounts of noise.
We present such further experiments in Appendix~\ref{app:noise}.

\section{Perspectives}\label{sec:perspectives}

We conclude with a list of open questions.
\begin{itemize}
\item \textbf{(Relaxing by +1)} Based on numerical simulations, we conjecture that relaxing to $\dimopt = \dimgt + 1$ yields a benign landscape for all ground truth configurations in the complete graph case.
  Is this truly the case, or is there a small pathological set of ground truths---perhaps of measure zero---for which the landscape admits spurious local minimizers?

A possible avenue toward sharpening Theorem~\ref{th:sqrt-n} in this direction is to optimize the weights placed on the eigenmodes in the construction of $S$ in~\eqref{defS}.
Another possible avenue is to obtain finer control on the spectrum of $P$ in the spirit of Lemma~\ref{traceP}.\footnote{We thank an anonymous reviewer for suggesting these directions.}
  
\item \textbf{(Incomplete graph)} Numerical experiments indicate that relaxing to $\dimopt$ only slightly more than $\dimgt$ improves the probability of recovery in some regimes.
  Can we provide a theoretical explanation---for example, for specific graph structures such as trilateration graphs, expander graphs, geometric graphs, or unit disk graphs?
  
\noindent Theorem~\ref{th:isotropic} is extended to (slightly) incomplete graphs by~\citet[\S 8.8]{Criscitiello2025thesis}. 
In particular, it is shown that if $\edges$ in~\eqref{eq:snl} is missing $m$ edges, $n$ is sufficiently large, and $\dimopt \gtrsim \kappa(1+m^3)(\dimgt + \log n)$, then~\eqref{eq:snl} has benign landscape with high probability.
Extending this result to substantially sparser graphs remains an open problem.

\item \textbf{(Noisy settings)} This paper assumes that distances are known \emph{exactly}, which is unrealistic for practical applications.
  Benign landscapes are robust to small perturbations, but without control on the scale of the perturbations.
  Can we still ensure benign landscapes for SNL under sufficiently large levels of measurement noise?
\item \textbf{(Isotropic case)} The numerical simulations in Section~\ref{sec:numerics} suggest that spurious minimizers are extremely rare for \emph{isotropic} Gaussian ground truths and the complete graph, even without relaxation.
  For standard Gaussian ground truths in $\reals^\dimgt$ with $\dimgt \geq 2$, can we prove that the landscape of~\eqref{upstairs}, \emph{without relaxation}, is benign with high probability when the number of points $n$ is sufficiently large?
\item \textbf{(Convergence rates)} This paper does not address rates of convergence.
  In particular, when the dimension is relaxed, problem~\eqref{eq:snl} typically loses quadratic growth around its set of minimizers.
  Is this an obstacle to fast convergence in the relaxed setting?  The recent work of~\citet{davis2024gradientdescentadaptivestepsize} provides a positive starting point for investigating this question.
\item \textbf{(Range-Aided SLAM)} Range-Aided SLAM is a variant of Simultaneous Localization and Mapping (SLAM) that incorporates distance measurements into the standard SLAM formulation~\citep{rosenrangeaidedslam2024}.
  Range-Aided SLAM blends elements of both SNL and synchronization of rotations, the latter of which also has been analyzed from a landscape perspective~\citep{mcrae2023benignlowdim}.
  What is the landscape of optimization problems arising in Range-Aided SLAM?
  Like in SNL, can one empirically improve performance by relaxing such problems?
\item \textbf{(Nonconvex variants.)} The localization problem admits several equivalent reformulations.
  \cite{halsted2022riemannian} propose an exact reformulation involving optimization over edge directions.
  Is the landscape of this optimization problem also benign after suitable relaxation?
  How are its critical configurations related to those of the SNL landscape studied in this paper?
\end{itemize}

\section*{Acknowledgements}

We thank Brighton Ancelin for helpful discussions regarding the counterexample in Proposition~\ref{originalcounterexample}.
We are also grateful to Frederike Dümbgen for insightful conversations about applications of SNL in robotics.
We thank Steven J.~Gortler for his enthusiastic feedback on an early draft, as well as for valuable discussions related to rigidity theory. 
This work was supported by the Swiss State Secretariat for Education, Research and Innovation (SERI) under contract number MB22.00027.

\bibliographystyle{plainnat}
\bibliography{references,boumal}

\appendix

\section{An alternative set of properties of the EDM map}\label{app:otherconditions}

In this section, we show that Theorem~\ref{th:sqrt-n} holds under a somewhat different set of properties than~\ref{P1} to~\ref{P5}.
Specifically, assume $\Delta$ satisfies~\ref{P1} to~\ref{P3}, and satisfies the following two properties in place of~\ref{P4} and~\ref{P5}:
\begin{enumerate} [label=\textbf{Q\arabic*}]
\setcounter{enumi}{3}
\item \label{P4prime} $\langle \ptgram, \Psi(\ptgram) \rangle \leq \frac{c}{2}\langle \ptgram, \Gamma(\ptgram) \rangle$ for all $\ptgram \in \Cent(n)$ satisfying $\trace(\ptgram) = 0$.

\item \label{P5prime} $\Gamma(\ptgram) = \sum_{i=1}^m a_i a_i^\top (a_i^\top \ptgram a_i)$ for all $\ptgram \in \Cent(n)$, for some collection of vectors $a_1, \ldots, a_m \in \reals^n$.
\end{enumerate}
For the EDM map, we have $c = n$ in~\ref{P4prime}.
Note that~\ref{P1} is a consequence of~\ref{P5prime}.

Observe that we only use properties~\ref{P4} and~\ref{P5} in the proof of Lemma~\ref{derivationdescent} (which relies on Lemma~\ref{lem:cute}).
Therefore, let us show how to prove Lemma~\ref{derivationdescent} with properties~\ref{P4prime} and~\ref{P5prime} instead (and without using Lemma~\ref{lem:cute}).  We follow the notation introduced in Section~\ref{sec:sqrt-n}.

Consider the descent direction $\dot \pt = \xi_i v^\top$ where $\xi_i$ is an eigenvector of $C$ with eigenvalue $\gamma_i \leq 0$, and $v = \Sigma^{1/2} u$ with $u \in \ker(V^\top \gtgram V)$.
The 1-criticality condition $C \pt = 0$ implies $\langle \pt v \xi_i^\top, \xi_i v^\top \pt^\top \rangle = 0$ and $\trace(\pt v \xi_i^\top + \xi_i v^\top \pt^\top) = 0$.
Plugging this descent direction into the 2-criticality condition~\eqref{eq:2crit} gives
  \begin{equation}\label{eq:tobemultiplied}
    \begin{split}
      0 &\leq \|v\|^2\langle \xi_i \xi_i^\top, C \rangle + \langle \pt v \xi_i^\top, (\Delta^* \circ \Delta)(\pt v \xi_i^\top + \xi_i v^\top \pt^\top) \rangle \\
        &= \|v\|^2\langle \xi_i \xi_i^\top, C \rangle + \langle \pt v \xi_i^\top, \pt v \xi_i^\top + \Psi(\pt v \xi_i^\top + \xi_i v^\top \pt^\top) \rangle \\
        &\stackrel{(1)}{\leq}
        \|v\|^2\langle \xi_i \xi_i^\top, C \rangle + \langle \pt v \xi_i^\top, \pt v \xi_i^\top + \frac{c}{2}\Gamma(\pt v \xi_i^\top + \xi_i v^\top \pt^\top) \rangle \\
        &\stackrel{(2)}{=}
        \|v\|^2\langle \xi_i \xi_i^\top, C \rangle + \langle \pt v \xi_i^\top, \pt v \xi_i^\top \rangle + c \langle \pt v v^\top \pt^\top, \Gamma(\xi_i \xi_i^\top) \rangle.
    \end{split}
  \end{equation}
We used~\ref{P4prime} for $\stackrel{(1)}{=}$, and~\ref{P5prime} for $\stackrel{(2)}{=}$.
Multiplying inequality~\eqref{eq:tobemultiplied} by $-\gamma_i \geq 0$, and summing the resulting inequalities from $i = 1, \ldots, \dimgt$ we obtain exactly equation~\eqref{eq:usefulguy2} in Lemma~\ref{derivationdescent}.
The rest of the proof of Lemma~\ref{derivationdescent} then proceeds in exactly the same way as in Section~\ref{sec:p2crit}.

\begin{remark}[Recipe for generating maps $\Delta$]
The following recipe generates maps satisfying~\ref{P1},~\ref{P2},~\ref{P3},~\ref{P4prime}, and~\ref{P5prime}.
Choose a collection of vectors $\{a_i\}_{i=1}^m$ so that $\sum_{i=1}^m \|a_i\|^4 < 1$ and $\spann(\{a_i a_i^\top\}_{i=1}^N) = \Cent(n)$ --- this latter condition ensures that the $\Gamma$ is a positive definite map.
Define $\Gamma(\ptgram) = \sum_{i=1}^N a_i a_i^\top (a_i^\top \ptgram a_i)$.
Then properties ~\ref{P1},~\ref{P2},~\ref{P3},~\ref{P5prime} hold, and property~\ref{P4prime} holds with some constant $c > 0$.
We then define $\Delta$ through
$$(\Delta^* \circ \Delta)^{-1}(\ptgram) = \ptgram - \Gamma(\ptgram) = \ptgram - \sum_{i=1}^N a_i a_i^\top (a_i^\top \ptgram a_i).$$
\end{remark}

\section{Derivation of the revealing formulas}\label{app:formulas}

Consider the important formulas stated in~\eqref{eq:deltastardelta}.
Composing equations~\eqref{EDMmap} and~\eqref{eq:adjoint}, we obtain
\begin{align*}
    (\Delta^* \circ \Delta)(\ptgram) &= \frac{1}{2}\bigg(2 \ptgram + n\Diag(\ptgram) + \trace(\ptgram) I - \big(\ones \diag(\ptgram)^\top + \diag(\ptgram) \ones^\top\big)\bigg) \\
    &= \ptgram + \frac{n}{2} \centeringmat\Diag(\ptgram)\centeringmat + \frac{1}{2}\trace(\ptgram) \centeringmat,
\end{align*}
assuming $\ptgram \in \Cent(n)$.

Using this formula, one then checks that the composition of $(\Delta^* \circ \Delta)$ and
$$\ptgram \mapsto \ptgram - \centeringmat \Diag(\ptgram) \centeringmat, \quad \quad \Cent(n) \to \Cent(n)$$
yields the identity, so this is the expression for $(\Delta^* \circ \Delta)^{-1}$.

Alternatively, the following formulas hold:
\begin{align*}
  \Delta^{-1}(\holmat) &= -\holmat + \frac{1}{n}\big(\ones \ones^\top \holmat + \holmat \ones \ones^\top\big) - \frac{\ones^\top \holmat \ones}{n^2} \ones \ones^\top = - \centeringmat \holmat \centeringmat,\\
  (\Delta^{-1})^{*}(\ptgram) &= (\Delta^{*})^{-1}(\ptgram) = -\ptgram + \Diag \ptgram.
\end{align*}
Composing these formulas also yields
\begin{align*}
 (\Delta^* \circ \Delta)^{-1}(\ptgram) &= \ptgram - \Diag(\ptgram) + \frac{1}{n}\big(\ones \diag(\ptgram)^\top + \diag(\ptgram) \ones^\top\big) - \frac{\trace \ptgram}{n^2} \ones \ones^\top \\
 &= \ptgram - \centeringmat \Diag(\ptgram) \centeringmat.
\end{align*}
We note that the above expression for $\Delta^{-1}$ is the basis of Schoenberg's Theorem~\citep{Schoenbergtheorem}.

\section{Eigenvalue interlacing supplementary}\label{app:interlacing}

The following interlacing theorem is classical and can be found in~\citep[Ch 4]{Horn1994}, for example.
\begin{proposition}[Eigenvalue Interlacing]\label{interlacing}
Let $M \in \Sym(\dimopt)$ have eigenvalues $m_1 \leq \ldots \leq m_\dimopt$.
Let $\calS$ be a $d$-dimensional subspace of $\reals^\dimopt$ with orthonormal basis $S \in \St(\dimopt, d)$.
Let $\tilde M = S^\top M S$ have eigenvalues $\tilde m_1 \leq \tilde m_2 \leq \ldots \leq \tilde m_d$.
Then,
\begin{align}\label{bunchofinterlacinginequalities}
  m_i \leq \tilde m_i \leq m_{\dimopt-d+i}, \quad \forall i \in [1,d].
\end{align}
\end{proposition}

Using the previous proposition, let us prove Lemma~\ref{equalityinterlacing1}, which additionally characterizes what happens when the interlacing inequalities become equalities.
\begin{proof}[Proof of Lemma~\ref{equalityinterlacing1}]
  Inequalities~\eqref{bunchofinterlacinginequalities} immediately imply~\eqref{conclusionequalitiesinterlacing}.
  It only remains to show $\dim(\calS \cap \ker(M)) \geq r-\dimopt+d$, assuming $\tilde m_1 = 0$.
  
  Let $M$ have associated eigenvectors $\mu_1, \ldots, \mu_\dimopt$.
  For $z \in \calS \cap \spann(\mu_1, \ldots, \mu_r)$, we can write $z = S \tilde z$, and so
  \begin{align*}
    0 = \tilde m_1 \leq \tilde z^\top \tilde M \tilde z = \tilde z^\top S^\top M S \tilde z = z^\top M z \leq m_r = 0.
  \end{align*}
  Hence, $0 = z^\top M z$, which implies $0 = M z$ (using that $\max_{z' \in \spann(\mu_1, \ldots, \mu_r)} z'^\top M z' \leq m_r = 0$).
  We conclude that
  $\calS \cap \spann(\mu_1, \ldots, \mu_r) \subseteq \calS \cap \ker(M).$
  
  On the other hand, 
  \begin{align}\label{supersmallhelper2}
  \dim(\calS \cap \spann(\mu_1, \ldots, \mu_r)) \geq \dim(\calS) + \dim(\spann(\mu_1, \ldots, \mu_r)) - \dimopt = d + r - \dimopt.
  \end{align}
Therefore, $\dim(\calS \cap \ker(M)) \geq r-\dimopt+d$.
\end{proof}

\section{Geometric interpretations of the descent directions}\label{app:geomdescdirections}

Given a 1-critical configuration $\pt$ that is not a minimizer, the goal of landscape analysis is to identify a \emph{descent direction} $\dot{\pt}$ satisfying $\langle \dot{\pt}, \nabla^2 \fu(\pt)[\dot{\pt}] \rangle < 0$, so that perturbing the configuration via $t \mapsto \pt + t \dot{\pt}$ leads to a decrease in the cost $\fu$.

In this section, we provide geometric interpretations of the descent directions used to prove Theorems~\ref{th:sqrt-n} and~\ref{thm:gt_dep_determ} in Sections~\ref{sec:sqrt-n} and~\ref{sec:isotropic}.

\paragraph{Descent directions in Theorem~\ref{th:sqrt-n}:}

By inspecting the proof of Lemma~\ref{derivationdescent} and the first proof of Lemma~\ref{lem:cute}, we find that the second-order criticality conditions~\eqref{eq:2crit} are evaluated along directions of the form $\dot{\pt} = \xi v^\top$, where $\xi$ is in $\reals^n$ and $v = \Sigma^{1/2} u$ for some $u \in \ker(V^\top \gtgram V)$.

Since the ground truth configuration $\gt$ and the critical configuration $\pt$ live in spaces of different dimensions, it is helpful to first map $\pt$ into the space of $\gt$. To this end, we define the ``best linear transformation'' $R \in \mathbb{R}^{\dimgt \times \dimopt}$ aligning $\pt$ to $\gt$ via a least-squares problem:\footnote{Since both $\pt$ and $\gt$ are centered, it is natural to consider only linear transformations rather than affine ones.}
\begin{align}\label{bestlinear}
  R := {\arg\min}_{R \in \mathbb{R}^{\dimgt \times \dimopt}} \sqfrobnorm{\gt - \pt R^\top} = \gt\transpose \pt (\pt\transpose \pt)^{-1}.
\end{align}
Using the decomposition $\pt = V \Sigma^{1/2}$, we can rewrite $R$ as
$R = \gt^\top \pt (\pt^\top \pt)^{-1} = \gt^\top V \Sigma^{-1/2}.$
It follows that $v = \Sigma^{1/2} u$ lies in the kernel of $R$ if and only if $u \in \ker(V^\top \gtgram V)$.

Let $\xi = (\xi^{(1)}, \ldots, \xi^{(n)})^\top \in \mathbb{R}^n$ denote the coordinates of $\xi$. We find that the descent direction $\dot{\pt} = \xi v^\top$ perturbs each point $z_i$ along the direction $v \in \ker(R)$ by an amount $\xi^{(i)}$, while leaving $R z_i$ unchanged: indeed, since $v \in \ker(R)$, we have $R(z_i + t v) = R z_i$ for all $t$.

Why might such a direction be useful? Intuitively, to decrease the objective $\fu$, one might hope to perturb $z_i$ in a way that moves $R z_i$ closer to $z_i^*$. Alternatively, from a more conservative standpoint, it is reasonable to at least avoid moving $R z_i$ further away from $z_i^*$. The descent direction achieves this latter goal by leaving $R z_i$ unchanged.

\paragraph{Descent directions in Theorem~\ref{thm:gt_dep_determ}:}

Inspecting the proof of Theorem~\ref{thm:gt_dep_determ} and the second proof of Lemma~\ref{lem:cute}, we find that the descent directions take the form $\dot{\pt} = (\gt - \pt R^\top) G$, where $G$ is a $\dimgt \times \dimopt$ matrix with i.i.d.\ standard Gaussian entries. We have shown that if $\pt$ is a 1-critical point but not a minimizer of $g$, then \emph{in expectation}, perturbing $\pt$ along $\dot{\pt}$ decreases the cost $g$.
Let us now interpret this descent direction geometrically.

Since the ground truth $\gt$ and the critical configuration $\pt$ live in spaces of different dimensions, it is natural to first embed $\gt$ into the space of $\pt$. A simple approach is to embed $\gt$ into a \emph{random} subspace $\im(G^\top) \subseteq \mathbb{R}^\dimopt$ by mapping each point $z_i^* \mapsto G^\top z_i^*$.

A naive strategy to try to decrease the cost would be to move each point $z_i$ towards the corresponding embedded ground truth point, i.e., to set $\dot{\pt} = \gt G - \pt$. However, this direct approach does not seem to work well.

Instead, the descent direction $\dot{\pt} = \gt G - \pt R^\top G$ used in Theorem~\ref{thm:gt_dep_determ} follows a slightly more refined procedure. It first maps the critical configuration $\pt$ into the subspace spanned by $\gt$ via the ``best linear transformation'' $R$ defined in~\eqref{bestlinear}, obtaining $\pt R^\top$. This aligned configuration is then embedded into the random subspace $\im(G^\top)$, resulting in $\pt R^\top G$. The descent direction thus measures the discrepancy between the embedded ground truth $\gt G$ and the aligned, embedded critical configuration $\pt R^\top G$.

\section{Robustness to noise}\label{app:noise}

In this section, we investigate the robustness of the numerical results from Section~\ref{sec:numerics} to small amounts of noise.
We follow the same procedure as in Section~\ref{sec:numerics} to generate ground truth configurations $\gt$, but we corrupt the squared distances with noise as follows.
The ground truth configurations yield Euclidean distance matrices $D = 2\Delta(\gt^{}\gt^\top)$, to which we add a random symmetric matrix $E \in \mathrm{Sym}(n)$ with i.i.d.\ Gaussian entries $E_{ij} \sim \mathcal{N}(0, 0.1)$.
This gives a noisy matrix $\tilde D = D + E$ of corrupted squared distances, and the optimization problem becomes
\begin{align}\label{eq:snl-corrupted}
  \min_{\pt \in \reals^{n \times \dimopt}} \, \fu(Z) = \big\|\Delta\big(\pt \pt^\top) - \tilde D/2\big\|_\frob^2,
\end{align}
which generally has a nonzero optimal value.

Because of the noise, the original ground truth configuration $\gt$ is typically not a minimizer of~\eqref{eq:snl-corrupted}.
Therefore, we run a trust-region algorithm initialized at $\gt$ and let it converge to a minimizer $\tilde \gt$ of~\eqref{eq:snl-corrupted}.
This solution serves as a proxy for the global minimizer of the corrupted problem, and we refer to it as \emph{the proxy} below.

We then repeat the same experimental setup as in Section~\ref{sec:numerics}:
for a given optimization rank $\dimopt \geq \dimgt$, we run a trust-region algorithm from a random initial configuration (consisting of $n$ independent standard Gaussian points in dimension $\dimopt$), observe the convergence point, and compare it to $\tilde \gt$.

\paragraph{Recovering the proxy for the global minimizer.}

In this first set of experiments, we define ``success'' as the ability to recover the proxy.
Specifically, we let the algorithm converge to $\pt \in \reals^{n \times \dimopt}$ from a random initialization, and say that the algorithm succeeds if, after projection to dimension $\dimgt$, the points $\pt$ match the proxy $\tilde \gt$ up to rigid motion symmetries, within a tolerance of $10^{-10} \sqrt{n\dimgt}$ on the norm of the difference.
Figure~\ref{fig:mds-norms-noise} shows the results of this experiment.
It would also be of interest to assess how well the solution $\pt$ found estimates the ground truth $\gt$, relative to statistical lower bounds; we do not do this here.

\begin{figure}[p]
  \centering
  \begin{subfigure}{0.49\textwidth}
    \centering
    \includegraphics{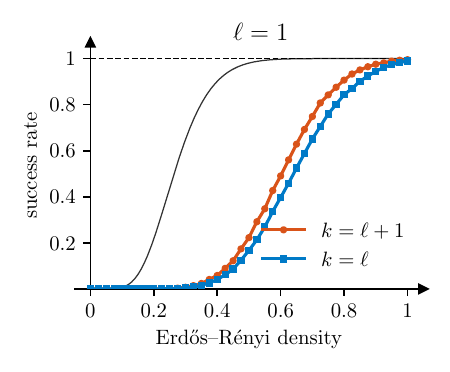}
    \caption{Simulations with $n = 10$ points.}
  \end{subfigure}
  \hfill
  \begin{subfigure}{0.49\textwidth}
    \centering
    \includegraphics{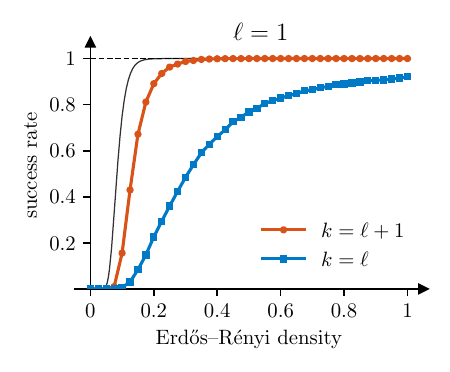}
    \caption{Simulations with $n = 50$ points.}
  \end{subfigure}
  \caption{Success rates for recovery of the proxy.}\label{fig:mds-norms-noise}
\end{figure}

\paragraph{Driving the cost to that of the proxy.}

In this second set of experiments, we define ``success'' as the ability to drive the cost to that of the proxy.
Specifically, we let the algorithm converge to $\pt \in \reals^{n \times \dimopt}$, and say that it succeeds if $\fu(\pt) \leq (1 + 10^{-10})\fu(\tilde \gt)$.
Note that this is a weaker success criterion than in the previous paragraph: other configurations than the proxy may also achieve a cost no larger than $\fu(\tilde \gt)$.
Figure~\ref{fig:mds-costs-noise} shows the results of this experiment.

\begin{figure}[p]
  \centering
  \begin{subfigure}{0.49\textwidth}
    \centering
    \includegraphics{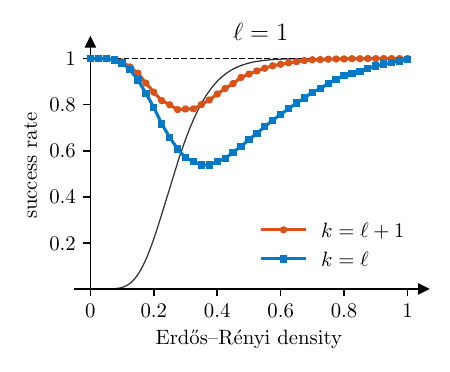}
    \caption{Simulations with $n = 10$ points.}
  \end{subfigure}
  \hfill
  \begin{subfigure}{0.49\textwidth}
    \centering
    \includegraphics{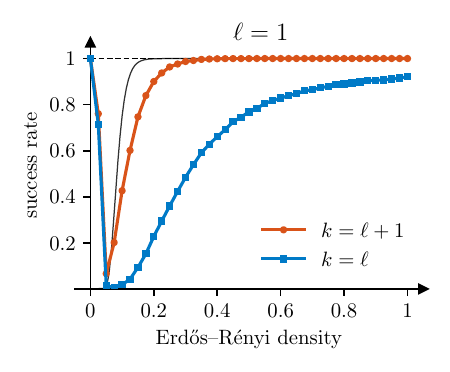}
    \caption{Simulations with $n = 50$ points.}
  \end{subfigure}
  \caption{Success rates for a cost lower than that of the proxy.}\label{fig:mds-costs-noise}
\end{figure}

\end{document}